\documentclass[11pt]{amsart}
\usepackage{amsthm,xcolor, bm}
\usepackage{latexsym,amssymb,amsfonts,amsmath,graphicx, url,mathtools}
\usepackage[vcentermath,enableskew]{youngtab}
\usepackage{color}
\usepackage[noadjust]{cite}
\usepackage{tikz,mathrsfs}
\usepackage[margin=1in]{geometry}
\usepackage{enumitem, bm}
\usepackage{dutchcal}
\usepackage{algpseudocode}
\usepackage{algorithm}
\usepackage[unicode,psdextra]{hyperref}

\newtheorem{theorem}{Theorem}[section]

\newtheorem{lemma}[theorem]{Lemma}
\newtheorem{conjecture}[theorem]{Conjecture}
\newtheorem{corollary}[theorem]{Corollary}

\theoremstyle{definition}
\newtheorem{definition}[theorem]{Definition}
\newtheorem{examplex}[theorem]{Example}
\newenvironment{example}
{\pushQED{\qed}\examplex}
{\popQED\endexamplex}

\newtheoremstyle{named}{}{}{\itshape}{}{\bfseries}{.}{.5em}{#1 \thmnote{#3}}
\theoremstyle{named}

\newcommand{\CG}{C^G} 
\newcommand{\CN}{C^N} 
\newcommand{\CL}{C^L} 
\newcommand{\CS}{C^S} 

\DeclareMathOperator{\cc}{\mathrm{invcode}} 
\newcommand{\occ}{\overline{\cc}} 

\DeclareMathOperator{\rr}{\mathrm{rajcode}} 
\newcommand{\orr}{\overline{\rr}} 

\newcommand{\LL}{\overline{L}}

\newcommand{\absrajcode}[1]{\mathrm{raj}(#1)} 
\newcommand{\absoverrajcode}[1]{\overline{\mathrm{raj}}(#1)} 
\newcommand{\absinvcode}[1]{\mathrm{inv}(#1)}

\DeclareMathOperator{\nested}{\mathrm{nested}}
\DeclareMathOperator{\greedy}{\mathrm{greedy}}
\DeclareMathOperator{\leap}{\mathrm{leap}}
\DeclareMathOperator{\stair}{\mathrm{stair}}
\DeclareMathOperator{\leads}{\mathrm{leads}}

\DeclareMathOperator{\LIS}{\mathrm{LIS}}
\DeclareMathOperator{\wt}{\mathrm{wt}}
\DeclareMathOperator{\dwt}{\overline{\mathrm{wt}}}

\title{On the degree of Grothendieck Polynomials}

\author{Matt Dreyer}
\address{Matt Dreyer, Department of Mathematics, Cornell University, Ithaca NY 14853.  \newline\textup{mjd367@cornell.edu}
}

\author{Karola M\'esz\'aros}
\address{Karola M\'esz\'aros, Department of Mathematics, Cornell University, Ithaca, NY 14853.  \newline\textup{karola@math.cornell.edu}
}

\author{Avery St.~Dizier}
\address{Avery St.~Dizier, Department of Mathematics, University of Illinois at Urbana-Champaign, Urbana, IL 61801.  \newline\textup{stdizie2@illinois.edu}}

\thanks{Karola M\'esz\'aros received support from  CAREER NSF Grant DMS-1847284. Avery St.~Dizier received support from NSF Grant DMS-2002079.}

\begin{document}
	\begin{abstract}
		A beautiful degree formula for the Grothendieck polynomials was recently given by Pechenik, Speyer, and Weigandt (2021). We provide an alternative proof of their degree formula, utilizing the climbing chain model for Grothendieck polynomials introduced by Lenart, Robinson, and Sottile (2006). Moreover, for any term order satisfying $x_1<x_2<\cdots<x_n$ we present the leading monomial of each homogeneous components of the Grothendieck polynomial $\mathfrak{G}_w(\bm{x})$, confirming a conjecture of Hafner (2022). We conclude with a conjecture for the leading monomials of the homogenegous components of $\mathfrak{G}_w(\bm{x})$ in any term order satisfying $x_1>x_2>\cdots>x_n$.
	\end{abstract}

	\maketitle
	
	\section{Introduction}
	Schubert polynomials and Grothendieck polynomials are multivariate polynomials associated to permutations in $S_n$. 
	Schubert and Grothendieck polynomials were introduced by Lascoux and Sch\"utzenberger \cite{LS1,LS2} in their study of the \emph{flag variety}, the parameter space of maximal sequences of nested vector spaces in $\mathbb{C}^n$. The flag variety admits a cell decomposition into \emph{Schubert varieties} indexed by permutations in $S_n$.
	
	Schubert (resp.\ Grothendieck) polynomials arise as a set of distinguished representatives for the cohomology (resp.\ K-theoretic) classes of Schubert varieties in the cohomology ring (resp.\ K-theory) of the flag variety. Schubert polynomials have a rich and well-studied combinatorial structure, admiting a myriad of formulas such as \cite{laddermoves, BJS, FKschub, nilcoxeter, thomas, lenart, prismtableaux,balancedtableaux}. Many formulas for Schubert polynomials generalize to Grothendieck polynomials as well. 
	
	Schubert and Grothendieck polynomials are geometrically natural choices of class representatives: Knutson and Miller \cite{grobner} showed them to be the multidegrees and K-polynomials respectively of \emph{matrix Schubert varieties}, determinantal varieties obtained by pulling the Schubert varieties in the flag variety back to $n\times n$ matrix space. See \cite[Chapter 15]{CCA} for a thorough introduction.
	
	In \cite{symmgrothdeg}, it was noted that Grothendieck polynomials give a combinatorial approach to studying the \emph{regularity} of matrix Schubert varieties, a measure of the complexity of their defining ideals. The regularity is given by the difference between the degrees of the Grothendieck and Schubert polynomial of $w\in S_n$. The Schubert polynomial is homogeneous of degree $\ell(w)$, the number of inversions. The degree of the Grothendieck polynomial was not well-understood at the time, and has since seen several developments.
	
	An explicit degree formula for Grothendieck polynomials of Grassmannian permutations was given in \cite{symmgrothdeg}. This was extended to vexillary permutations in \cite{vexgrothdeg} by Rajchgot, Robichaux, and Weigandt; an alternate proof given by Hafner in \cite{elena1}. A general explicit degree formula for any Grothendieck polynomial was given by Pechenik, Speyer, and Weigandt in \cite{PSW}, in the form of permutation statistics called the \emph{Rajchgot code} $\rr(w)$ and \emph{Rajchgot index} $\absrajcode{w}$:	
	\begin{theorem}[{\cite[Theorem 1.1]{PSW}}]
		\label{thm:pswleadingterm}
		For any $w\in S_n$,
		\[\deg \mathfrak{G}_w=\absrajcode{w}.\]
		Moreover in any term order satisfying $x_1<x_2<\cdots<x_n$, the leading monomial of the highest-degree homogeneous component $\mathfrak{G}_w^{\mathrm{top}}$ of $\mathfrak{G}_w$ is $\bm{x}^{\rr(w)}$.
	\end{theorem}
	The authors of \cite{PSW} expressed frustration (\cite[Remark 7.2]{PSW}) that their proof of Theorem \ref{thm:pswleadingterm} does not provide an explicit combinatorial representative of the monomial $\bm{x}^{\rr(w)}$ in the Grothendieck polynomial. We provide an explicit combinatorial representative in Definition \ref{def:nestedchain}, and alternate proof of Theorem \ref{thm:pswleadingterm} via Theorem \ref{thm:leadingchainweight}. In this paper we restrict our attention to Grothendieck polynomials in $x$ variables only. 
	
	Our main tool is the climbing chain model for Schubert and Grothendieck polynomials, introduced by Lenart, Robinson, and Sottile in \cite{LRS}. The model expresses the Schubert and Grothendieck polynomials in terms of certain saturated chains in the Bruhat order on permutations. We introduce explicit chains representing the leading monomials of $\mathfrak{S}_w$ and $\mathfrak{G}_w^{\mathrm{top}}$ in term orders with $x_1<x_2<\cdots<x_n$. We conjecturally do the same for term orders with $x_1>x_2>\cdots>x_n$ in Section \ref{sec:>}.

	We also investigate the leading monomials of the other homogeneous components of Grothendieck polynomials, a question first posed by Weigandt \cite{annacascade}. A potential answer was conjectured by Hafner:	
	\begin{conjecture}[{\cite[Conjecture 4.1]{elena1}}]
		\label{conj:elena}
		Fix $w\in S_n$ and any term order with $x_1<x_2<\cdots<x_n$. For $\ell(w)< k\leq \absrajcode{w}$, let $m_k(\bm{x})$ be the leading monomial of the degree $k$ homogeneous component of $\mathfrak{G}_w$. Then
		\[m_k(\bm{x}) = x_pm_{k-1}(\bm{x}) \]
		where $p$ is the largest index such that $x_p m_{k-1}(\bm{x})$ divides $\bm{x}^{\rr(w)}$.
	\end{conjecture}
	
	We address Hafner's conjecture by first providing climbing chains for her conjectured monomials (Definition \ref{def:interpolatingchains} and Theorem \ref{thm:interpolatingchainsexponents}). We use these chains to verify Conjecture \ref{conj:elena} (Theorem \ref{thm:interpolatingleadingterms}). We formulate an analogue of Conjecture \ref{conj:elena} for term orders with $x_1>x_2>\cdots>x_n$ (Conjecture \ref{conj:reverseinterpolating}).
	
	\medskip
	
	\noindent \textbf{Roadmap of the paper.}
	Section \ref{sec:back} contains general background.
	Section \ref{sec:PSW} is a summary of the results of \cite{PSW}.
	Section \ref{sec:climb} is a presentation of the Lenart, Robinson, and Sottile climbing chain model for Grothendieck polynomials (\cite{LRS}) via Rothe diagrams. 
	 Section \ref{sec:greedy} details the construction of a chain yielding the monomial $\bm{x}^{\cc(w)}$ in $\mathfrak{G}_w(\bm{x})$. 
	 Section \ref{sec:nested} details the construction of a chain yielding the monomial $\bm{x}^{\rr(w)}$ in $\mathfrak{G}_w(\bm{x})$. 
	In Section \ref{sec:raj} we give a new proof that the degree of the Grothendieck polynomial $\mathfrak{G}_w(\bm{x})$ is the Rajchgot index $\absrajcode{w}$. 
	In Section \ref{sec:rajcode}, we complete the alternate proof of Theorem \ref{thm:pswleadingterm}.
	In Section \ref{sec:inter} we define chains that interpolate between the greedy and the nested chains. We show the terms predicted by Hafner as the leading monomials of the homogenegous components of $\mathfrak{G}_w(\bm{x})$ lie in the support of $\mathfrak{G}_w(\bm{x})$. 
	In Section \ref{sec:lead}, we prove Conjecture \ref{conj:elena}.
	Section \ref{sec:>} is devoted to term orders satisfying $x_1>x_2>\cdots>x_n$. In Section \ref{sec:leap} we construct the leaping chain which yields the leading term of the lowest degree homogenegous components of $\mathfrak{G}_w(\bm{x})$. 
	In Section \ref{sec:stair} we construct the staircase chain, which we prove yields a monomial of degree $\absrajcode{w}$ and conjecture is the leading monomial of $\mathfrak{G}_w^{\mathrm{top}}$. We conclude by conjecturing a dual version of Conjecture \ref{conj:elena} for term orders satisfying $x_1>x_2>\cdots>x_n$.

	\section{Background} 
	\label{sec:back}
	
	\subsection{Conventions}
	For $n\in \mathbb{N}$, we write $[n]$ to mean the set $\{1,2,\ldots,n \}$. For $1\leq i<j\leq n$ we write $t_{ij}$ for the transposition in the symmetric group $S_n$ swapping $i$ and $j$. We write $s_j$ for the adjacent transposition $t_{j,j+1}$.
	
	Throughout, we write permutations in one-line notation (as a string) $w=w(1)w(2)\cdots w(n)$. We will take permutations as acting on the right -- switching positions, not values. For example $ws_1$ equals $w$ with the numbers $w(1)$ and $w(2)$ swapped. 
	
	For $v\in \mathbb{Z}^n$, we write $|v|$ for $v_1+v_2+\cdots+v_n$. We denote the standard basis vectors of $\mathbb{Z}^n$ by $e_1,e_2,\ldots,e_n$. We use the notation $\overline{v}$ to denote the \emph{vector complement} 
	\[\overline{v}_k = n-k - v_k \mbox{ for each } k\in[n].\]
	
	The \emph{leading term} of a polynomial $f\in \mathbb{Z}[x_1,\ldots,x_n]$ is the term of the largest monomial under $<$ appearing in $f$. The \emph{leading monomial} is the leading term divided by its coefficient. For example if $f(x_1,x_2)=x_1^2+2x_1x_2+3x_2^2$ and $x_1<x_2$, the leading term of $f$ is $3x_2^2$ and the leading monomial is $x_2^2$.
	
	\subsection{Schubert and Grothendieck Polynomials}
	
	\begin{definition}
		Fix any $n\geq 0$. The \emph{divided difference operators} $\partial_j$ for $j\in[n-1]$ are operators on the polynomial ring $\mathbb{Z}[x_1,\ldots,x_n]$ defined by
		\[\partial_j(f)=\frac{f-(s_j\cdot f)}{x_j-x_{j+1}}
		=\frac{f(x_1,\ldots,x_n)-f(x_1,\ldots,x_{j-1},x_{j+1},x_j,x_{j+2},\ldots,x_n)}{x_j-x_{j+1}}.\]
		The \emph{isobaric divided difference operators} $\overline{\partial}_j$ are defined on $\mathbb{Z}[x_1,\ldots,x_n]$ by
		$\overline{\partial}_j(f)=\partial_j(f-x_{j+1}f)$.
	\end{definition}
	
	\begin{definition}
		The \emph{Schubert polynomial} $\mathfrak{S}_w$ of $w\in S_n$ is defined recursively on the weak Bruhat order. Let $w_0=n \hspace{.1cm} n\!-\!1 \hspace{.1cm} \cdots \hspace{.1cm} 2 \hspace{.1cm} 1 \in S_n$, the longest permutation in $S_n$. If $w\neq w_0$ then there is $j\in [n-1]$ with $w(j)<w(j+1)$ (called an \emph{ascent} of $w$). The polynomial $\mathfrak{S}_w$ is defined by
		\begin{align*}
			\mathfrak{S}_w=\begin{cases}
				x_1^{n-1}x_2^{n-2}\cdots x_{n-1}&\mbox{ if } w=w_0,\\
				\partial_j \mathfrak{S}_{ws_j} &\mbox{ if } w(j)<w(j+1).
			\end{cases}
		\end{align*}
	\end{definition}
	
	\begin{definition}
		The \emph{Grothendieck polynomial} $\mathfrak{G}_w$ of $w\in S_n$ is defined analogously to the Schubert polynomial,
		with
		\begin{align*}
			\mathfrak{G}_w=\begin{cases}
				x_1^{n-1}x_2^{n-2}\cdots x_{n-1}&\mbox{ if } w=w_0,\\
				\overline{\partial}_j \mathfrak{G}_{ws_j} &\mbox{ if } w(j)<w(j+1).
			\end{cases}
		\end{align*}	
	\end{definition}
	
	Recall the number of inversions of $w\in S_n$ is $\ell(w)=\#\{(i,j)\mid 1\leq i<j\leq n \mbox{ with } w(i)>w(j) \}.$ It can be seen from the recursive definitions above that $\mathfrak{S}_w$ is homogeneous of degree $\ell(w)$, and equals the lowest-degree nonzero homogeneous component of $\mathfrak{G}_w$. See \cite{manivel} for a deeper introduction to Schubert polynomials. 
	
	Associated to each permutation $w\in S_n$ is its Rothe diagram $D(w)$.
	
	\begin{definition}
		The \emph{Rothe diagram} of a permutation $w\in S_n$ is the set
		\[D(w)=\{(i,j):\, 1\leq i,j\leq n,\,\, w(i)>j,\,\, \mbox{and } w^{-1}(j)>i \}.\]
	\end{definition}
	
	We will make use of the following graphical interpretation of $D(w)$. Start with an $n\times n$ grid of boxes with the usual matrix indexing. Place a dot in box $(i,w(i))$ for each $i\in [n]$. Draw rays emanating south and east of each dot, called \emph{hooks}. Remove any boxes hit by a hook. The remaining boxes lie exactly in the positions $D(w)$.
	
	\begin{definition}
		The \emph{Lehmer code} of $w\in S_n$ is the vector $\cc(w)=(\cc(w)_1,\ldots,\cc(w)_n)$ where $\cc(w)_i$ equals the number of boxes in row $i$ of $D(w)$. Symbolically,
		\[\cc(w)_i = \#\{j\mid i<j\mbox{ and } w(j)<w(i) \}. \]
	\end{definition}
	
	\begin{example}
		\label{exp:31452diagram}
		If $w=31452$, then $D(w)=\{(1,1),(1,2),(3,2),(4,2) \}$. We draw $D(w)$ by placing east and south hooks at each point $(i, w(i))$ in the $5\times 5$ grid and coloring all boxes not hit by any hook:
		\begin{center}
			\begin{tikzpicture}[scale=.55]
				\draw (0,0)--(5,0)--(5,5)--(0,5)--(0,0);
				\draw[draw=red] (5,4.5) -- (2.5,4.5) node[red] {$\bullet$} -- (2.5,0);
				\draw[draw=red] (5,3.5) -- (0.5,3.5) node[red] {$\bullet$} -- (0.5,0);
				\draw[draw=red] (5,1.5) -- (4.5,1.5) node[red] {$\bullet$} -- (4.5,0);
				\draw[draw=red] (5,2.5) -- (3.5,2.5) node[red] {$\bullet$} -- (3.5,0);
				\draw[draw=red] (5,0.5) -- (1.5,0.5) node[red] {$\bullet$} -- (1.5,0);
				
				\filldraw[draw=black,fill=lightgray] (0,4)--(1,4)--(1,5)--(0,5)--(0,4);
				\filldraw[draw=black,fill=lightgray] (1,4)--(2,4)--(2,5)--(1,5)--(1,4);
				\filldraw[draw=black,fill=lightgray] (1,2)--(2,2)--(2,3)--(1,3)--(1,2);
				\filldraw[draw=black,fill=lightgray] (1,1)--(2,1)--(2,2)--(1,2)--(1,1);
				
				\draw (0,1)--(5,1);
				\draw (0,2)--(5,2);
				\draw (0,3)--(5,3);
				\draw (0,4)--(5,4);
				
				\draw (1,0)--(1,5);
				\draw (2,0)--(2,5);
				\draw (3,0)--(3,5);
				\draw (4,0)--(4,5);
				
				\node at (-1.5,2.5) {$D(w)=$};
				
				\node at (5.3,2.5) {.};
			\end{tikzpicture}
		\end{center}
		The Lehmer code of $w$ is $\cc(w) = (2,0,1,1,0)$.
	\end{example}
	
	\section{Degree of Grothendieck Polynomials}
	\label{sec:PSW}
	
	We recall the results of Pechenik, Speyer, and Weigandt \cite{PSW} giving the leading monomial of the highest degree component of $\mathfrak{G}_w$ in any term order satisfying $x_1< x_2<\cdots< x_n$.
	
	\begin{definition}\label{def:1}
		An \emph{increasing subsequence} of a permutation $w\in S_n$ starting from position $q\in [n]$ is a vector $\alpha=(\alpha_1,\ldots,\alpha_m)$ such that
		\[q=\alpha_1<\alpha_2<\cdots<\alpha_m\leq n \quad \mbox{and}\quad w(\alpha_1)<w(\alpha_2)<\cdots<w(\alpha_m).\] 
		An increasing subsequence of a permutation $w$ starting from position $q$ is called \emph{longest} if there is no longer such sequence, i.e.\ $m$ is maximal. We refer to $m$ as the \emph{length} of $\alpha$. We denote by $\LIS(w,q)$ the set of all longest increasing subsequences of $w$ starting from position $q$.
	\end{definition}
	
	Graphically, increasing subsequences starting from position $q$ of a permutation $w$ are exactly sequences of nested hooks in the Rothe diagram $D(w)$ starting with the hook in row $q$.
	
	\begin{definition}
		To each $w\in S_n$, associate the vector 
		\[\orr(w)=(\orr(w)_1,\ldots,\orr(w)_n)\in\mathbb{Z}^n,\]
		where for each $i$, $\orr(w)_i$ equals the length of any $\alpha\in \LIS(w,i)$ minus one. 
	\end{definition}
	
	\begin{definition}[{\cite{PSW}}]
		The \emph{Rajchgot code} of $w\in S_n$ is the vector $\rr(w)\in \mathbb{Z}^n$ with 
		\[\rr(w)_k = n-k - \orr(w)_k \quad\mbox{for each } k\in [n]. \]
	\end{definition}

	Observe that the entries of $\rr(w)$ count the complement of longest increasing sequences: for any fixed $\alpha\in\LIS(w,k)$, $\rr(w)_k$ equals the number of elements in $\{k,k+1,\ldots,n\}$ that did not appear in $\alpha$.
	
	\begin{definition}
		The \emph{Rajchgot index} of $w\in S_n$ is
		\[\absrajcode{w} = |\rr(w)| = \sum_{k=1}^n \rr(w)_k.\]
		Similarly, let $\absoverrajcode{w} = |\orr(w)|$.
	\end{definition}

	\begin{example}
		Consider $w=48513726$, with Rothe diagram
		\begin{center}
			\begin{tikzpicture}[scale=.55]
				\draw (0,0)--(8,0)--(8,8)--(0,8)--(0,0);
				
				\draw[draw=red] (8,7.5) -- (3.5,7.5) node[red] {$\bullet$} -- (3.5,0); 
				\draw[draw=red] (8,6.5) -- (7.5,6.5) node[red] {$\bullet$} -- (7.5,0); 
				\draw[draw=red] (8,5.5) -- (4.5,5.5) node[red] {$\bullet$} -- (4.5,0); 
				\draw[draw=red] (8,4.5) -- (0.5,4.5) node[red] {$\bullet$} -- (0.5,0); 
				\draw[draw=red] (8,3.5) -- (2.5,3.5) node[red] {$\bullet$} -- (2.5,0); 
				\draw[draw=red] (8,2.5) -- (6.5,2.5) node[red] {$\bullet$} -- (6.5,0); 
				\draw[draw=red] (8,1.5) -- (1.5,1.5) node[red] {$\bullet$} -- (1.5,0); 
				\draw[draw=red] (8,0.5) -- (5.5,0.5) node[red] {$\bullet$} -- (5.5,0); 
				
				\filldraw[draw=black,fill=lightgray] (0,7)--(1,7)--(1,8)--(0,8)--(0,7); 
				\filldraw[draw=black,fill=lightgray] (1,7)--(2,7)--(2,8)--(1,8)--(1,7); 
				\filldraw[draw=black,fill=lightgray] (2,7)--(3,7)--(3,8)--(2,8)--(2,7); 
				\filldraw[draw=black,fill=lightgray] (0,6)--(1,6)--(1,7)--(0,7)--(0,6); 
				\filldraw[draw=black,fill=lightgray] (1,6)--(2,6)--(2,7)--(1,7)--(1,6); 
				\filldraw[draw=black,fill=lightgray] (2,6)--(3,6)--(3,7)--(2,7)--(2,6); 
				\filldraw[draw=black,fill=lightgray] (4,6)--(5,6)--(5,7)--(4,7)--(4,6); 
				\filldraw[draw=black,fill=lightgray] (5,6)--(6,6)--(6,7)--(5,7)--(5,6); 
				\filldraw[draw=black,fill=lightgray] (6,6)--(7,6)--(7,7)--(6,7)--(6,6); 
				\filldraw[draw=black,fill=lightgray] (0,5)--(1,5)--(1,6)--(0,6)--(0,5); 
				\filldraw[draw=black,fill=lightgray] (1,5)--(2,5)--(2,6)--(1,6)--(1,5); 
				\filldraw[draw=black,fill=lightgray] (2,5)--(3,5)--(3,6)--(2,6)--(2,5); 
				\filldraw[draw=black,fill=lightgray] (1,3)--(2,3)--(2,4)--(1,4)--(1,3); 
				\filldraw[draw=black,fill=lightgray] (1,2)--(2,2)--(2,3)--(1,3)--(1,2); 
				\filldraw[draw=black,fill=lightgray] (5,2)--(6,2)--(6,3)--(5,3)--(5,2); 
				
				\draw (0,1)--(8,1);
				\draw (0,2)--(8,2);
				\draw (0,3)--(8,3);
				\draw (0,4)--(8,4);
				\draw (0,5)--(8,5);
				\draw (0,6)--(8,6);
				\draw (0,7)--(8,7);
				
				\draw (1,0)--(1,8);
				\draw (2,0)--(2,8);
				\draw (3,0)--(3,8);
				\draw (4,0)--(4,8);
				\draw (5,0)--(5,8);
				\draw (6,0)--(6,8);
				\draw (7,0)--(7,8);
				
				\node at (-1.5,4.5) {$D(w)=$};
				\node at (8.3,4.5) {.};
			\end{tikzpicture}
		\end{center}
		The longest increasing sequences in $w$ are
		\begin{center}
			\begin{tabular}{lcl}
				$\LIS(w,1)$=$\{(1, 3, 6),(1, 3, 8) \}$,&\phantom{aaaa}&$\LIS(w,2)=\{(2)\}$,\\
				$\LIS(w,3)$=$\{(3, 6),(3, 8)\}$,&\phantom{aaaa}&$\LIS(w,4)=\{(4, 5, 6),(4, 5, 8),(4, 7, 8)\}$,\\
				$\LIS(w,5)$=$\{(5, 6),(5, 8)\}$,&\phantom{aaaa}&$\LIS(w,6)=\{(6)\}$,\\
				$\LIS(w,7)$=$\{(7,8)\}$,&\phantom{aaaa}&$\LIS(w,8)=\{(8)\}$.
			\end{tabular}
		\end{center}
	Observe that these sequences show up graphically as the longest sequences of (row indices of) nested hooks in $D(w)$.
	One computes
	\begin{align*}
		\orr(w)&=(2,0,1,2,1,0,1,0),\\
		\absoverrajcode{w} &= 2+0+1+2+1+0+1+0 = 7,\\
		\rr(w) &= (7,6,5,4,3,2,1,0)-\orr(w)=(5, 6, 4, 2, 2, 2, 0, 0),\quad \mbox{and}\\
		\absrajcode{w}&=5+6+4+2+2+2+0+0 = 21.\qedhere
	\end{align*}
	\end{example}

	In general, $\rr(w)_k\geq \cc(w)_k$ for all $k\in [n]$. This follows since $\cc(w)_k$ counts the number of inversions of the form $(k,\star)$ in $w$, and $\rr(w)_k$ counts the number of entries in $[k+1,n]$ excluded in any (fixed) sequence $\alpha\in \LIS(w,k)$.
	
	\section{Climbing Chains and Markings} 
	\label{sec:climb}
	
	We recall climbing chains, a combinatorial model for Schubert and Grothendieck polynomials due to Lenart, Robinson, and Sottile in \cite{LRS}. We depict their construction in terms of the Rothe diagrams; this diverges from their exposition, but we find this rendering of climbing chains particularly helpful. 
	
	The \emph{(strong) Bruhat order} is the partial order $\leq$ on $S_n$ defined as the transitive closure of the relations 
	\[w\leq wt_{ij} \]
	for any $i<j$ such that $w(i)<w(j)$.
	We write $w\lessdot v$ to denote a cover relation in the Bruhat order. The following well-known lemma describes all cover relations of the Bruhat order. See for instance \cite[Lemma 2.1.4]{combcoxgroups} for a proof.
	
	\begin{lemma}
		\label{lem:cover}
		For $v,w\in S_n$, $w\lessdot v$ if and only if there is $i,j\in [n]$ with $i<j$ such that:
		\begin{itemize}
			\item $v=wt_{ij}$,
			\item $w(i)<w(j)$,
			\item $\{k \mid i<k<j\}\cap \{k\mid w(i)<w(k)<w(j) \}=\emptyset$.
		\end{itemize}
	\end{lemma}

	Lemma \ref{lem:cover} can be interpreted graphically as follows. For $i<j$, $w\lessdot wt_{ij}$ if and only if inside $D(w)$:
	\begin{itemize}
		\item The dot at $(i,w(i))$ lies north and west of the dot at $(j,w(j))$,
		\item No other dot lies in the region bounded by the hooks at $(i,w(i))$ and $(j,w(j))$.
	\end{itemize}

	\begin{example}
		If $w=31452$ (continuing Example \ref{exp:31452diagram}), then $w\lessdot wt_{ij}$ exactly when 
		\[(i,j)\in \{(1,3),(2,3),(2,5),(3,4)\}.\qedhere\]
	\end{example}

	\begin{definition}
		A \emph{climbing chain} of $w\in S_n$ is a sequence $C$ of pairs $(i_1,j_1),\ldots,(i_m,j_m)$, called \emph{links}, such that 
		\begin{itemize}
			\item $i_p<j_p$ for each $p\in [m]$, 
			\item $i_1\leq i_2\leq \cdots \leq i_m$,
			\item $wt_{i_1j_1}\cdots t_{i_mj_m} = w_0$,
			\item $wt_{i_1j_1}\cdots t_{i_pj_p} \lessdot wt_{i_1j_1}\cdots t_{i_{p+1}j_{p+1}}$ for each $p\in [m]$ (in the Bruhat order on $S_n$).
		\end{itemize}
		We call $\ell(C)=m$ the \emph{length} of $C$, and will write $C_p$ to indicate the link $(i_p,j_p)$ for each $p\in [m]$.
	\end{definition}
	Climbing chains encode special saturated chains in the Bruhat order on $S_n$ from a given permutation $w$ to $w_0$, the longest permutation in $S_n$.
	
	\begin{example}
		\label{exp:31452unmarked}
		Let $w=256341$, so
		\begin{center}
			\begin{tikzpicture}[scale=.55]
			\draw (0,0)--(6,0)--(6,6)--(0,6)--(0,0);
			\draw[draw=red] (6,5.5) -- (1.5,5.5) node[red] {$\bullet$} -- (1.5,0);
			\draw[draw=red] (6,4.5) -- (4.5,4.5) node[red] {$\bullet$} -- (4.5,0);
			\draw[draw=red] (6,3.5) -- (5.5,3.5) node[red] {$\bullet$} -- (5.5,0);
			\draw[draw=red] (6,2.5) -- (2.5,2.5) node[red] {$\bullet$} -- (2.5,0);
			\draw[draw=red] (6,1.5) -- (3.5,1.5) node[red] {$\bullet$} -- (3.5,0);
			\draw[draw=red] (6,0.5) -- (0.5,0.5) node[red] {$\bullet$} -- (0.5,0);
			
			\filldraw[draw=black,fill=lightgray] (0,5)--(1,5)--(1,6)--(0,6)--(0,5);
			\filldraw[draw=black,fill=lightgray] (0,4)--(1,4)--(1,5)--(0,5)--(0,4);
			\filldraw[draw=black,fill=lightgray] (0,3)--(1,3)--(1,4)--(0,4)--(0,3);
			\filldraw[draw=black,fill=lightgray] (0,2)--(1,2)--(1,3)--(0,3)--(0,2);
			\filldraw[draw=black,fill=lightgray] (0,1)--(1,1)--(1,2)--(0,2)--(0,1);
			\filldraw[draw=black,fill=lightgray] (2,4)--(3,4)--(3,5)--(2,5)--(2,4);
			\filldraw[draw=black,fill=lightgray] (3,4)--(4,4)--(4,5)--(3,5)--(3,4);
			\filldraw[draw=black,fill=lightgray] (2,3)--(3,3)--(3,4)--(2,4)--(2,3);
			\filldraw[draw=black,fill=lightgray] (3,3)--(4,3)--(4,4)--(3,4)--(3,3);
			
			\draw (0,1)--(6,1);
			\draw (0,2)--(6,2);
			\draw (0,3)--(6,3);
			\draw (0,4)--(6,4);
			\draw (0,5)--(6,5);
			
			\draw (1,0)--(1,6);
			\draw (2,0)--(2,6);
			\draw (3,0)--(3,6);
			\draw (4,0)--(4,6);
			\draw (5,0)--(5,6);
			
			\node at (-1.5,2.5) {$D(w)=$};
			\node at (6.3,2.5) {.};
			\end{tikzpicture}
		\end{center}
	
		We construct a climbing chain of $w$. Since $w(1)\neq 6$, the first link of any climbing chain of $w$ will be of the form $(1,j_1)$, where the hook in row $j_1$ of $D(w)$ sits south and east of the hook in row 1, and no hook sits between them. The available choices are $j_1\in \{2,4\}$. We will choose $j_1=4$.
		
		Observe that $D(wt_{14}) = D(356241)$ is obtained from $D(w)$ by swapping the column indices of the dots in rows 1 and 4. That is,
	
		\begin{center}
			\begin{tikzpicture}[scale=.55]
				\draw (0,0)--(6,0)--(6,6)--(0,6)--(0,0);
				\draw[draw=red] (6,5.5) -- (2.5,5.5) node[red] {$\bullet$} -- (2.5,0);
				\draw[draw=red] (6,4.5) -- (4.5,4.5) node[red] {$\bullet$} -- (4.5,0);
				\draw[draw=red] (6,3.5) -- (5.5,3.5) node[red] {$\bullet$} -- (5.5,0);
				\draw[draw=red] (6,2.5) -- (1.5,2.5) node[red] {$\bullet$} -- (1.5,0);
				\draw[draw=red] (6,1.5) -- (3.5,1.5) node[red] {$\bullet$} -- (3.5,0);
				\draw[draw=red] (6,0.5) -- (0.5,0.5) node[red] {$\bullet$} -- (0.5,0);
				
				\filldraw[draw=black,fill=lightgray] (0,5)--(1,5)--(1,6)--(0,6)--(0,5);
				\filldraw[draw=black,fill=lightgray] (1,5)--(2,5)--(2,6)--(1,6)--(1,5);
				\filldraw[draw=black,fill=lightgray] (0,4)--(1,4)--(1,5)--(0,5)--(0,4);
				\filldraw[draw=black,fill=lightgray] (0,3)--(1,3)--(1,4)--(0,4)--(0,3);
				\filldraw[draw=black,fill=lightgray] (0,2)--(1,2)--(1,3)--(0,3)--(0,2);
				\filldraw[draw=black,fill=lightgray] (0,1)--(1,1)--(1,2)--(0,2)--(0,1);
				\filldraw[draw=black,fill=lightgray] (1,4)--(2,4)--(2,5)--(1,5)--(1,4);
				\filldraw[draw=black,fill=lightgray] (3,4)--(4,4)--(4,5)--(3,5)--(3,4);
				\filldraw[draw=black,fill=lightgray] (1,3)--(2,3)--(2,4)--(1,4)--(1,3);
				\filldraw[draw=black,fill=lightgray] (3,3)--(4,3)--(4,4)--(3,4)--(3,3);
				
				\draw (0,1)--(6,1);
				\draw (0,2)--(6,2);
				\draw (0,3)--(6,3);
				\draw (0,4)--(6,4);
				\draw (0,5)--(6,5);
				
				\draw (1,0)--(1,6);
				\draw (2,0)--(2,6);
				\draw (3,0)--(3,6);
				\draw (4,0)--(4,6);
				\draw (5,0)--(5,6);
				
				\node at (-2,2.5) {$D(wt_{14})=$};
				\node at (6.3,2.5) {.};
			\end{tikzpicture}
		\end{center}
		Since $wt_{14}(1)\neq 6$, the next link of the chain must be of the form $(1,j_2)$ where the hook in row $j_2$ of $D(wt_{14})$ sits south and east of the hook in row 1, and no hook sits between them. The available choices are $j_2\in \{2,5\}$. We will choose $j_2=2$.
		
		Note $wt_{14}t_{12} = 536241$, and has Rothe diagram
		
		\begin{center}
			\begin{tikzpicture}[scale=.55]
				\draw (0,0)--(6,0)--(6,6)--(0,6)--(0,0);
				\draw[draw=red] (6,5.5) -- (4.5,5.5) node[red] {$\bullet$} -- (4.5,0);
				\draw[draw=red] (6,4.5) -- (2.5,4.5) node[red] {$\bullet$} -- (2.5,0);
				\draw[draw=red] (6,3.5) -- (5.5,3.5) node[red] {$\bullet$} -- (5.5,0);
				\draw[draw=red] (6,2.5) -- (1.5,2.5) node[red] {$\bullet$} -- (1.5,0);
				\draw[draw=red] (6,1.5) -- (3.5,1.5) node[red] {$\bullet$} -- (3.5,0);
				\draw[draw=red] (6,0.5) -- (0.5,0.5) node[red] {$\bullet$} -- (0.5,0);
				
				\filldraw[draw=black,fill=lightgray] (0,5)--(1,5)--(1,6)--(0,6)--(0,5);
				\filldraw[draw=black,fill=lightgray] (1,5)--(2,5)--(2,6)--(1,6)--(1,5);
				\filldraw[draw=black,fill=lightgray] (0,4)--(1,4)--(1,5)--(0,5)--(0,4);
				\filldraw[draw=black,fill=lightgray] (0,3)--(1,3)--(1,4)--(0,4)--(0,3);
				\filldraw[draw=black,fill=lightgray] (0,2)--(1,2)--(1,3)--(0,3)--(0,2);
				\filldraw[draw=black,fill=lightgray] (0,1)--(1,1)--(1,2)--(0,2)--(0,1);
				\filldraw[draw=black,fill=lightgray] (1,4)--(2,4)--(2,5)--(1,5)--(1,4);
				\filldraw[draw=black,fill=lightgray] (3,5)--(4,5)--(4,6)--(3,6)--(3,5);
				\filldraw[draw=black,fill=lightgray] (2,5)--(3,5)--(3,6)--(2,6)--(2,5);
				\filldraw[draw=black,fill=lightgray] (1,3)--(2,3)--(2,4)--(1,4)--(1,3);
				\filldraw[draw=black,fill=lightgray] (3,3)--(4,3)--(4,4)--(3,4)--(3,3);
				
				\draw (0,1)--(6,1);
				\draw (0,2)--(6,2);
				\draw (0,3)--(6,3);
				\draw (0,4)--(6,4);
				\draw (0,5)--(6,5);
				
				\draw (1,0)--(1,6);
				\draw (2,0)--(2,6);
				\draw (3,0)--(3,6);
				\draw (4,0)--(4,6);
				\draw (5,0)--(5,6);
				
				\node at (-2.3,2.5) {$D(wt_{14}t_{12})=$};
				\node at (6.3,2.5) {.};
			\end{tikzpicture}
		\end{center}
		
	Since $wt_{14}t_{12}(1) \neq 6$, the next link of this chain has to be of the form $(1,j_3)$, and the only available choice is $j_3=3$. Then the next permutation is $wt_{14}t_{12}t_{13}=635241$, with diagram
	
	\begin{center}
		\begin{tikzpicture}[scale=.55]
		\draw (0,0)--(6,0)--(6,6)--(0,6)--(0,0);
		\draw[draw=red] (6,5.5) -- (5.5,5.5) node[red] {$\bullet$} -- (5.5,0);
		\draw[draw=red] (6,4.5) -- (2.5,4.5) node[red] {$\bullet$} -- (2.5,0);
		\draw[draw=red] (6,3.5) -- (4.5,3.5) node[red] {$\bullet$} -- (4.5,0);
		\draw[draw=red] (6,2.5) -- (1.5,2.5) node[red] {$\bullet$} -- (1.5,0);
		\draw[draw=red] (6,1.5) -- (3.5,1.5) node[red] {$\bullet$} -- (3.5,0);
		\draw[draw=red] (6,0.5) -- (0.5,0.5) node[red] {$\bullet$} -- (0.5,0);
		
		\filldraw[draw=black,fill=lightgray] (0,5)--(1,5)--(1,6)--(0,6)--(0,5);
		\filldraw[draw=black,fill=lightgray] (1,5)--(2,5)--(2,6)--(1,6)--(1,5);
		\filldraw[draw=black,fill=lightgray] (0,4)--(1,4)--(1,5)--(0,5)--(0,4);
		\filldraw[draw=black,fill=lightgray] (0,3)--(1,3)--(1,4)--(0,4)--(0,3);
		\filldraw[draw=black,fill=lightgray] (0,2)--(1,2)--(1,3)--(0,3)--(0,2);
		\filldraw[draw=black,fill=lightgray] (0,1)--(1,1)--(1,2)--(0,2)--(0,1);
		\filldraw[draw=black,fill=lightgray] (1,4)--(2,4)--(2,5)--(1,5)--(1,4);
		\filldraw[draw=black,fill=lightgray] (3,5)--(4,5)--(4,6)--(3,6)--(3,5);
		\filldraw[draw=black,fill=lightgray] (4,5)--(5,5)--(5,6)--(4,6)--(4,5);
		\filldraw[draw=black,fill=lightgray] (2,5)--(3,5)--(3,6)--(2,6)--(2,5);
		\filldraw[draw=black,fill=lightgray] (1,3)--(2,3)--(2,4)--(1,4)--(1,3);
		\filldraw[draw=black,fill=lightgray] (3,3)--(4,3)--(4,4)--(3,4)--(3,3);
		
		\draw (0,1)--(6,1);
		\draw (0,2)--(6,2);
		\draw (0,3)--(6,3);
		\draw (0,4)--(6,4);
		\draw (0,5)--(6,5);
		
		\draw (1,0)--(1,6);
		\draw (2,0)--(2,6);
		\draw (3,0)--(3,6);
		\draw (4,0)--(4,6);
		\draw (5,0)--(5,6);
		
		\node at (-2.5,2.5) {$D(wt_{14}t_{12}t_{13})=$};
		\node at (6.3,2.5) {.};
		\end{tikzpicture}
	\end{center}
	
	Finally, $wt_{14}t_{12}t_{13}(1)=6$. Since $wt_{14}t_{12}t_{13}(2)\neq 5$, the next link of the chain must be of the form $(2,j_4)$. The available choices are $j_4\in \{3,5\}$. Continuing in this fashion, one may obtain either of the chains
	\begin{align*}
	C^{(1)} &= ((1,4),(1,2),(1,3),(2,3),(3,5),(4,5)),\\
	C^{(2)} &= ((1,4),(1,2),(1,3),(2,5),(2,3),(4,5)).
	\end{align*}
	
	With different earlier choices, one can obtain any of the chains
	\begin{align*}
		C^{(3)} &= ((1,2),(1,3),(2,3),(3,4),(3,5),(4,5)),\\
		C^{(4)} &= ((1,2),(1,3),(2,4),(2,3),(3,5),(4,5)),\\
		C^{(5)} &= ((1,2),(1,3),(2,4),(2,5),(2,3),(4,5)),\\
		C^{(6)} &= ((1,4),(1,5),(1,2),(1,3),(2,3),(4,5)).\\
	\end{align*}
	These are the six climbing chains of $w=256341$.
	\end{example}

	Observe that all climbing chains of a given permutation $w$ have the same length $\ell(w_0)-\ell(w)$.
	
	\begin{definition}
		Associated to each climbing chain $C$ of a permutation $w\in S_n$ is a set of special links $M(C)$ called the \emph{minimal markings} of $C$. As $C$ is a sequence of distinct pairs $(i_p,j_p)$, we will abuse notation slightly and write $M(C)\subseteq C$. If the links of $C$ are $C_p=(i_p,j_p)$ for each $p\in [\ell(C)]$, then (taking $i_0=0$)
		\[M(C)=\{C_p \mid i_{p-1}<i_{p} \mbox{ or } i_{p-1}=i_p \mbox{ and } j_{p-1}<j_p\}. \]
	\end{definition}

	Observe that in the trivial case $w=w_0$, the only climbing chain is the empty sequence. We take the empty sequence to have minimal markings $\emptyset$.

	\begin{definition}
		For $w\in S_n$, a \emph{marked climbing chain} of $w$ is a pair $(C,U)$ where $C$ is a climbing chain of $w$ and $U$ is a superset of the minimal markings of $C$, that is $M(C)\subseteq U\subseteq C$.
	\end{definition}

	\begin{example}
		\label{exp:31452marked}
		Continuing Example \ref{exp:31452unmarked}, let $w=256341$. Using overlines to denote markings, the climbing chains of $w$ together with their minimal markings are
		\begin{align*}
			&(\overline{(1,4)},{(1,2)},\overline{(1,3)},\overline{(2,3)},\overline{(3,5)},\overline{(4,5)}),\\
			&(\overline{(1,4)},{(1,2)},\overline{(1,3)},\overline{(2,5)},{(2,3)},\overline{(4,5)}),\\
			&(\overline{(1,2)},\overline{(1,3)},\overline{(2,3)},\overline{(3,4)},\overline{(3,5)},\overline{(4,5)}),\\
			&(\overline{(1,2)},\overline{(1,3)},\overline{(2,4)},{(2,3)},\overline{(3,5)},\overline{(4,5)}),\\
			&(\overline{(1,2)},\overline{(1,3)},\overline{(2,4)},\overline{(2,5)},{(2,3)},\overline{(4,5)}),\\
			&(\overline{(1,4)},\overline{(1,5)},{(1,2)},\overline{(1,3)},\overline{(2,3)},\overline{(4,5)}).\qedhere
		\end{align*}
	\end{example}

	\begin{definition}
		The \emph{dual weight} of a marked climbing chain $\xi=(C,U)$ of $w\in S_n$ is the vector \[\dwt(\xi)=(\dwt(\xi)_1,\ldots,\dwt(\xi)_n) \]
		where $\dwt(\xi)_k$ equals the number of links $(k,\star)\in U$.
	\end{definition}
	
	\begin{definition}
		The \emph{weight} of a marked climbing chain $\xi=(C,U)$ of $w\in S_n$ is the vector \[\wt(\xi)=(\wt(\xi)_1,\ldots,\wt(\xi)_n)\] whose $k$th component equals $n-k-\dwt(\xi)_k$.
	\end{definition}
	
	\begin{theorem}[{\cite[Theorem 5.6]{LRS}}]
		\label{thm:grothformula}
		For any $w\in S_n$,
		\[\mathfrak{G}_w(x_1,\ldots,x_n) = \sum_{\xi} \mathrm{sign}(\xi) \bm{x}^{\wt(\xi)} =\sum_{\xi} \mathrm{sign}(\xi) \frac{\mathfrak{G}_{w_0}}{\bm{x}^{\dwt(\xi)}}, \]
		where the sum is over all marked climbing chains $\xi=(C,U)$ of $w$, and $\mathrm{sign}(\xi) = (-1)^{\ell(C)-\#U}$.
	\end{theorem}
	
	\begin{corollary}[\cite{LRS}]
		For any $w\in S_n$,
		\[\mathfrak{S}_w(x_1,\ldots,x_n) = \sum_{\xi} \bm{x}^{\wt(\xi)} = \sum_{\xi} \frac{\mathfrak{G}_{w_0}}{\bm{x}^{\dwt(\xi)}}, \]
		where the sum is over all marked climbing chains $\xi=(C,U)$ of $w$ with $U=C$.
	\end{corollary}

	\begin{example}
		\label{exp:31452polys}
		Continuing Example \ref{exp:31452marked}, let $w=256341$. Then the Schubert and Grothendieck polynomials of $w$ are given by
		\begin{align*}
			\mathfrak{S}_w &=\frac{x_1^5x_2^4x_3^3x_4^2x_5}{x_1^3x_2x_3x_4}+
							 \frac{x_1^5x_2^4x_3^3x_4^2x_5}{x_1^3x_2^2x_4}+
							 \frac{x_1^5x_2^4x_3^3x_4^2x_5}{x_1^2x_2x_3^2x_4}+
							 \frac{x_1^5x_2^4x_3^3x_4^2x_5}{x_1^2x_2^2x_3x_4}+
							 \frac{x_1^5x_2^4x_3^3x_4^2x_5}{x_1^2x_2^3x_4}+
							 \frac{x_1^5x_2^4x_3^3x_4^2x_5}{x_1^4x_2x_4}\\
			\mathfrak{G}_w &= \mathfrak{S}_w\\
			&\qquad-\frac{x_1^5x_2^4x_3^3x_4^2x_5}{x_1^3x_2x_3x_4}-
					\frac{x_1^5x_2^4x_3^3x_4^2x_5}{x_1^2x_2^2x_4}-
					\frac{x_1^5x_2^4x_3^3x_4^2x_5}{x_1^3x_2x_4}-
					\frac{x_1^5x_2^4x_3^3x_4^2x_5}{x_1^2x_2x_3x_4}-
					\frac{x_1^5x_2^4x_3^3x_4^2x_5}{x_1^2x_2^2x_4}-
					\frac{x_1^5x_2^4x_3^3x_4^2x_5}{x_1^3x_2x_4}
			\\&\qquad+\frac{x_1^5x_2^4x_3^3x_4^2x_5}{x_1^2x_2x_4}.\qedhere
		\end{align*}
	\end{example}

	\begin{lemma}
		\label{lem:markingrecursion}
		Fix $w\in S_n$ with $w\neq w_0$. Let $C$ be any climbing chain of $w$. Suppose $C_p=(i_p,j_p)$ for $p\in [\ell(C)]$, and set $w'=wt_{i_1j_1}$. Define $C'$ to equal $C$ with $C_1$ removed. Then $C'$ is a climbing chain for $w'$ and 
		\[
		M(C) = 
		\begin{cases}
		(M(C')\setminus\{C_2\})\cup \{C_1\}&\mbox{if } \ell(C)>1 \mbox{ and }C_2\notin M(C),\\
		M(C')\cup \{C_1\}&\mbox{otherwise}.
		\end{cases} 
		\]		
		Consequently,
		\[
			\dwt(C,M(C)) = 
			\begin{cases}
				\dwt(C',M(C'))&\mbox{if } \ell(C)>1 \mbox{ and }C_2\notin M(C),\\
				\dwt(C',M(C'))+e_{i_1}&\mbox{otherwise}.
			\end{cases}
		\] 
	\end{lemma}
	\begin{proof}
		It is straightforward to see from the definition that $C'$ is a climbing chain for $w'$. 
		The first equality follows immediately from the local restrictions defining the minimal markings of a climbing chain. The second equality follows from the first by counting markings.
	\end{proof}
	
		\section{The Greedy Chain}\label{sec:greedy}
		
	We construct a particular climbing chain which we refer to as the greedy chain. We show that the greedy chain yields the monomial $\bm{x}^{\cc(w)}$. In particular, this monomial is the leading term of the Schubert polynomial in any term order satisfying $x_1<x_2<\cdots<x_n$. 
	
	\begin{definition}
		To each permutation $w\in S_n$ we associate a pair of integers $\greedy(w)$, called the \emph{greedy pair} of $w$, as follows. The greedy pair of $w_0$ is undefined. If $w\neq w_0$, set 
		\[a=\min\{k\mid w(k)\neq n+1-k\}\quad \mbox{and}\quad b=w^{-1}(w(i)+1).\] 
		 Then $\greedy(w)=(a,b)$.
	\end{definition}

	\begin{definition}
		Define the \emph{greedy climbing chain} $\CG(w)$ of $w\in S_n$ as follows. If $w=w_0$, then $\CG(w)$ is the empty sequence. Otherwise, let $\greedy(w)=(a,b)$. Define $\CG(w)$ inductively by prepending $(a,b)$ to $\CG(wt_{ab})$.
	\end{definition}
	
	The following example demonstrates a graphical interpretation of $\greedy(w)$.
	\begin{example}
		Continuing Example \ref{exp:31452marked} with $w=256341$, we compute $\CG(w)$.
		Recall the Rothe diagram of $w$ is
		\begin{center}
			\begin{tikzpicture}[scale=.55]
			\draw (0,0)--(6,0)--(6,6)--(0,6)--(0,0);
			\draw[draw=red] (6,5.5) -- (1.5,5.5) node[red] {$\bullet$} -- (1.5,0);
			\draw[draw=red] (6,4.5) -- (4.5,4.5) node[red] {$\bullet$} -- (4.5,0);
			\draw[draw=red] (6,3.5) -- (5.5,3.5) node[red] {$\bullet$} -- (5.5,0);
			\draw[draw=red] (6,2.5) -- (2.5,2.5) node[red] {$\bullet$} -- (2.5,0);
			\draw[draw=red] (6,1.5) -- (3.5,1.5) node[red] {$\bullet$} -- (3.5,0);
			\draw[draw=red] (6,0.5) -- (0.5,0.5) node[red] {$\bullet$} -- (0.5,0);
			
			\filldraw[draw=black,fill=lightgray] (0,5)--(1,5)--(1,6)--(0,6)--(0,5);
			\filldraw[draw=black,fill=lightgray] (0,4)--(1,4)--(1,5)--(0,5)--(0,4);
			\filldraw[draw=black,fill=lightgray] (0,3)--(1,3)--(1,4)--(0,4)--(0,3);
			\filldraw[draw=black,fill=lightgray] (0,2)--(1,2)--(1,3)--(0,3)--(0,2);
			\filldraw[draw=black,fill=lightgray] (0,1)--(1,1)--(1,2)--(0,2)--(0,1);
			\filldraw[draw=black,fill=lightgray] (2,4)--(3,4)--(3,5)--(2,5)--(2,4);
			\filldraw[draw=black,fill=lightgray] (3,4)--(4,4)--(4,5)--(3,5)--(3,4);
			\filldraw[draw=black,fill=lightgray] (2,3)--(3,3)--(3,4)--(2,4)--(2,3);
			\filldraw[draw=black,fill=lightgray] (3,3)--(4,3)--(4,4)--(3,4)--(3,3);
			
			\draw (0,1)--(6,1);
			\draw (0,2)--(6,2);
			\draw (0,3)--(6,3);
			\draw (0,4)--(6,4);
			\draw (0,5)--(6,5);
			
			\draw (1,0)--(1,6);
			\draw (2,0)--(2,6);
			\draw (3,0)--(3,6);
			\draw (4,0)--(4,6);
			\draw (5,0)--(5,6);
			
			\node at (-1.5,2.5) {$D(w)=$};
			\node at (6.3,2.5) {.};
			\end{tikzpicture}
		\end{center}
		Let $\greedy(w)=(i_1,j_1)$. Since $w(1)\neq 6$, $i_1=1$. The definition of the greedy pair says that \[j_1=w^{-1}(w(i_1)+1),\] that is $j_1$ is the row index of the dot in column $w(i_1)+1$. Graphically, $j_1$ is the row index of the leftmost dot inside the region south and east of the hook in row $i_1$. Thus, $j_1=4$. 
		
		The Rothe diagram of $wt_{14}$ is
		\begin{center}
			\begin{tikzpicture}[scale=.55]
			\draw (0,0)--(6,0)--(6,6)--(0,6)--(0,0);
			\draw[draw=red] (6,5.5) -- (2.5,5.5) node[red] {$\bullet$} -- (2.5,0);
			\draw[draw=red] (6,4.5) -- (4.5,4.5) node[red] {$\bullet$} -- (4.5,0);
			\draw[draw=red] (6,3.5) -- (5.5,3.5) node[red] {$\bullet$} -- (5.5,0);
			\draw[draw=red] (6,2.5) -- (1.5,2.5) node[red] {$\bullet$} -- (1.5,0);
			\draw[draw=red] (6,1.5) -- (3.5,1.5) node[red] {$\bullet$} -- (3.5,0);
			\draw[draw=red] (6,0.5) -- (0.5,0.5) node[red] {$\bullet$} -- (0.5,0);
			
			\filldraw[draw=black,fill=lightgray] (0,5)--(1,5)--(1,6)--(0,6)--(0,5);
			\filldraw[draw=black,fill=lightgray] (1,5)--(2,5)--(2,6)--(1,6)--(1,5);
			\filldraw[draw=black,fill=lightgray] (0,4)--(1,4)--(1,5)--(0,5)--(0,4);
			\filldraw[draw=black,fill=lightgray] (0,3)--(1,3)--(1,4)--(0,4)--(0,3);
			\filldraw[draw=black,fill=lightgray] (0,2)--(1,2)--(1,3)--(0,3)--(0,2);
			\filldraw[draw=black,fill=lightgray] (0,1)--(1,1)--(1,2)--(0,2)--(0,1);
			\filldraw[draw=black,fill=lightgray] (1,4)--(2,4)--(2,5)--(1,5)--(1,4);
			\filldraw[draw=black,fill=lightgray] (3,4)--(4,4)--(4,5)--(3,5)--(3,4);
			\filldraw[draw=black,fill=lightgray] (1,3)--(2,3)--(2,4)--(1,4)--(1,3);
			\filldraw[draw=black,fill=lightgray] (3,3)--(4,3)--(4,4)--(3,4)--(3,3);
			
			\draw (0,1)--(6,1);
			\draw (0,2)--(6,2);
			\draw (0,3)--(6,3);
			\draw (0,4)--(6,4);
			\draw (0,5)--(6,5);
			
			\draw (1,0)--(1,6);
			\draw (2,0)--(2,6);
			\draw (3,0)--(3,6);
			\draw (4,0)--(4,6);
			\draw (5,0)--(5,6);
			
			\node at (-2,2.5) {$D(wt_{14})=$};
			\node at (6.3,2.5) {.};
			\end{tikzpicture}
		\end{center}
		From the Rothe diagram, we observe that $\greedy(wt_{14}) = (1,5)$. Continuing in this fashion yields 
		\[\CG(w)=((1,4),(1,5),(1,2),(1,3),(2,3),(4,5)). \]
		This is the chain $C^{(6)}$ in Example \ref{exp:31452unmarked}.
	\end{example}

	The following lemma is straightforward and is left to the reader.
	\begin{lemma} 
		\label{lem:greedyischain} 
		For each $w\in S_n$, $\CG(w)$ is a climbing chain of $w$.
	\end{lemma}

	\begin{lemma}
		\label{lem:crecursion}
		Fix $w\in S_n$ with $w\neq w_0$. Suppose $\greedy(w)=(a,b)$ and set $w'=wt_{ab}$. Then
		\[
		\cc(w) = \cc(w')-e_a.
		\]
	\end{lemma}
	\begin{proof}
		In $D(w)$, swapping the hooks in rows $a$ and $b$, that is columns $w(a)$ and $w(b)$, swaps all boxes in those columns, and adds a new box at position $(a,w(a))$.
	\end{proof}

	\begin{theorem}
		\label{thm:lehmerequality}		
		Fix any $w\in S_n$. Set $\xi(w)$ to be the chain $\CG(w)$ with all links marked, that is \[\xi(w)=(\CG(w),\CG(w)).\] Then \[\wt(\xi(w)) = \cc(w). \]
	\end{theorem}
	\begin{proof}
		Since $i_1=\min\{k\mid w(k)\neq n+1-k\}$, it follows that $\cc(w)_k = n-k$ for $k<i_1$, and $\cc(w)_{i_1} = w(i_1)-1$. The greedy chain will move the hook in row $i_1$ to the immediate next column until it hits column $n-i_1+1$. The hook starts in column $w(i_1)$, so this will take $n-i_1+1-w(i_1) = n-i_1-\cc(w)_{i_1}$ many steps. At this point,
		the greedy chain will start moving the hook in a later row. Iterating this argument, the greedy chain will take $n-i_p-\cc(w)_{i_p}$ many steps of the form $(i_p,\star)$ for any $i_p$ that it takes such a step. 
		
		When the greedy chain takes a step $(i_p,\star)$, Lemma \ref{lem:crecursion} shows that it does not change the number of boxes in more southern rows. Suppose the greedy chain does not take a step of the form $(i_p,\star)$. The last time it takes a step in a more northern row, the number of boxes in row $i_p$ is still the same as in $D(w)$. But the greedy chain then proceeds to a more southern row, so the number of boxes in row $i_p$ must already be $n-i_p$. Thus, $\cc(w)_{i_p}=n-i_p$, and the greedy chain does indeed take $0=n-i_p-\cc(w)_{i_p}$ many steps of the form $(i_p,\star)$.
	\end{proof}

	\section{The Nested Chain} \label{sec:nested}
	
	We construct a particular climbing chain which we refer to as the nested chain. We show that the nested chain yields the monomial $\bm{x}^{\rr(w)}$. We show in Section \ref{sec:rajcode} that this monomial is the leading monomial of the highest-degree homogeneous component of the Grothendieck polynomial in any term order satisfying $x_1<x_2<\cdots<x_n$, giving an alternate proof of Theorem \ref{thm:pswleadingterm}.

	\begin{definition}
		To each $w\in S_n$ we associate a pair of integers $\nested(w)$, called the \emph{nested pair} of $w$, as follows. The nested pair of $w_0$ is undefined.
		If $w\neq w_0$, set $q=\min\{j\mid w(j)\neq n+1-j\}$. Let $\alpha$ be the lexicographically last element of $\LIS(w,q)$ (Definition \ref{def:1}). Then $\nested(w)=(\alpha_1,\alpha_2)$.
	\end{definition}
	
	\begin{definition}
		\label{def:nestedchain}
		Define the \emph{nested climbing chain} $\CN(w)$ of $w\in S_n$ as follows. If $w=w_0$, then $\CN(w)$ is the empty sequence. Otherwise, let $\nested(w)=(a,b)$. Define $\CN(w)$ inductively by prepending $(a,b)$ to $\CN(wt_{ab})$. 
	\end{definition}

	The following example gives a graphical interpretation for $\nested(w)$ in terms of nested hooks in $D(w)$.
	
	\begin{example}
		Continuing Example \ref{exp:31452marked} with $w=256341$, we compute $\CN(w)$, $M(\CN(w))$, and $\orr(w)$. For reference, the Rothe diagram of $w$ is
		\begin{center}
			\begin{tikzpicture}[scale=.55]
			\draw (0,0)--(6,0)--(6,6)--(0,6)--(0,0);
			\draw[draw=red] (6,5.5) -- (1.5,5.5) node[red] {$\bullet$} -- (1.5,0);
			\draw[draw=red] (6,4.5) -- (4.5,4.5) node[red] {$\bullet$} -- (4.5,0);
			\draw[draw=red] (6,3.5) -- (5.5,3.5) node[red] {$\bullet$} -- (5.5,0);
			\draw[draw=red] (6,2.5) -- (2.5,2.5) node[red] {$\bullet$} -- (2.5,0);
			\draw[draw=red] (6,1.5) -- (3.5,1.5) node[red] {$\bullet$} -- (3.5,0);
			\draw[draw=red] (6,0.5) -- (0.5,0.5) node[red] {$\bullet$} -- (0.5,0);
			
			\filldraw[draw=black,fill=lightgray] (0,5)--(1,5)--(1,6)--(0,6)--(0,5);
			\filldraw[draw=black,fill=lightgray] (0,4)--(1,4)--(1,5)--(0,5)--(0,4);
			\filldraw[draw=black,fill=lightgray] (0,3)--(1,3)--(1,4)--(0,4)--(0,3);
			\filldraw[draw=black,fill=lightgray] (0,2)--(1,2)--(1,3)--(0,3)--(0,2);
			\filldraw[draw=black,fill=lightgray] (0,1)--(1,1)--(1,2)--(0,2)--(0,1);
			\filldraw[draw=black,fill=lightgray] (2,4)--(3,4)--(3,5)--(2,5)--(2,4);
			\filldraw[draw=black,fill=lightgray] (3,4)--(4,4)--(4,5)--(3,5)--(3,4);
			\filldraw[draw=black,fill=lightgray] (2,3)--(3,3)--(3,4)--(2,4)--(2,3);
			\filldraw[draw=black,fill=lightgray] (3,3)--(4,3)--(4,4)--(3,4)--(3,3);
			
			\draw (0,1)--(6,1);
			\draw (0,2)--(6,2);
			\draw (0,3)--(6,3);
			\draw (0,4)--(6,4);
			\draw (0,5)--(6,5);
			
			\draw (1,0)--(1,6);
			\draw (2,0)--(2,6);
			\draw (3,0)--(3,6);
			\draw (4,0)--(4,6);
			\draw (5,0)--(5,6);
			
			\node at (-1.5,2.5) {$D(w)=$};
			\node at (6.3,2.5) {.};
			\end{tikzpicture}
		\end{center} 
		We have $\LIS(w,1)=\{(1,2,3),(1,4,5)\}$, so $\CN(w)_1=(1,4)$. Graphically, this amounts to finding the southmost sequence among all longest sequences of hooks nested under $(1,w(1))$.		
		
		Next, $wt_{14} = 356241$, so $\LIS(wt_{14},1) = \{(1,2,3)\}$ and $\CN(w)_2=(1,2)$. Then, $wt_{14}t_{12} = 536241$, so $\LIS(wt_{14}t_{12},1) = \{(1,3)\}$ and $\CN(w)_3=(1,3)$. Continuing in this fashion yields 
		\[\CN(w)=((1,4),(1,2),(1,3),(2,5),(2,3),(4,5)). \]
		This is exactly the chain $C^{(2)}$ from Example \ref{exp:31452unmarked}. One also  observes 
		\begin{align*}
			M(\CN(w)) &= \{(1,4),(1,3),(2,5),(4,5)\}, \mbox{ and}\\
			\orr(w) &= (2,1,0,1,0,0) = \dwt(C^N(w),M(\CN(w))).
		\end{align*}
	\end{example}
	
	\begin{lemma}
		\label{lem:nestedischain}
		For each $w\in S_n$, $\CN(w)$ is a climbing chain of $w$.
	\end{lemma}
	\begin{proof}
		If $w=w_0$, there is nothing to prove. Otherwise, let $\CN(w)_p=(i_p,j_p)$ for each $p$ and set $w'=wt_{i_1j_1}$. We show that $w\lessdot w'$. If not, then there is $p$ such that $i_1<p<j_1$ and $w(i_1)<w(p)<w(j_1)$. However, this would contradict the choice of $j_1$. 
		
		We now prove the lemma by induction on $m=\ell(w_0)-\ell(w)$. If $m=0$, then $w=w_0$ and the result is trivial. Assume the lemma holds for all $m'<m$. Suppose $\ell(w)=\ell(w_0)-m$. Since $\CN(w')$ is a truncation of $\CN(w)$, by induction it is enough to check that $i_1\leq i_2$, which follows easily from the definition of $\CN(w)$.
	\end{proof}
	
	 \begin{figure}[hb]
	 	\begin{center}
	 		\includegraphics[scale=1]{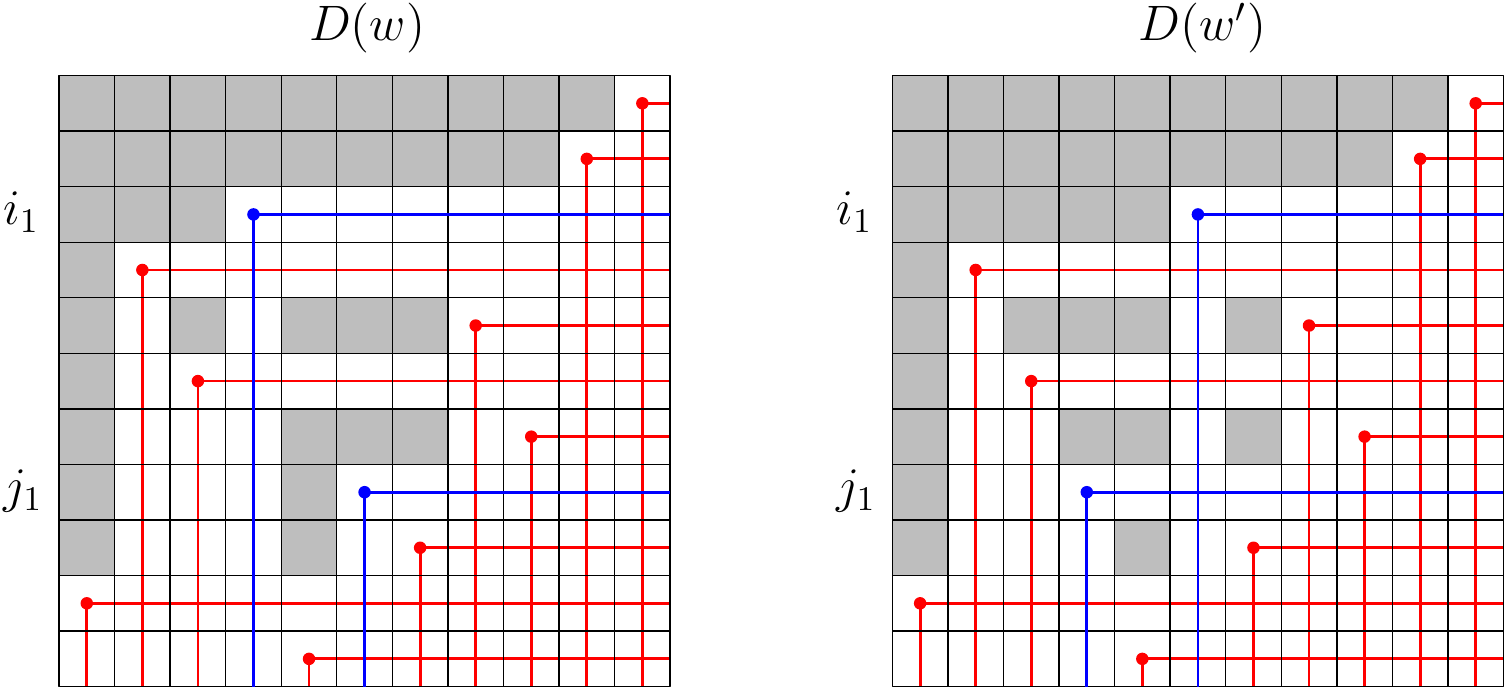}
	 	\end{center}
	 	\caption{A visual aid for the structure of $w$ and $w'$ in the proof of Theorem \ref{thm:drecursion}.}
	 	\label{fig:bargraphs}
	 \end{figure}
 
	We now prove a recursive formula for the vectors $\orr(w)$. In Figure \ref{fig:bargraphs}, we offer a visual aid for the technical arguments used to prove the recursion.
	
	\begin{theorem}
		\label{thm:drecursion}
		Fix $w\in S_n$ with $w\neq w_0$. 
		Suppose $\CN(w)_p=(i_p,j_p)$ for $p\in [\ell(\CN(w))]$, and set $w'=wt_{i_1j_1}$. Then
		\[
		\orr(w) = 
		\begin{cases}
		\orr(w') &\mbox{if } \ell(\CN(w))>1 \mbox{ and }(i_2,j_2)\notin M(\CN(w)),\\
		\orr(w')+e_{i_1}&\mbox{otherwise}.
		\end{cases} 
		\]
	\end{theorem}
	\begin{proof}
		By the definition of climbing chains, $i_1$ is the earliest position in which $w$ differs from $w_0$. Then for $k<i_1$, we have $w(k) = w'(k) = n+1-k$. Thus, $\orr(w)_k=0=\orr(w')_k$ whenever $k<i_1$. Additionally, since $w$ and $w'$ agree after position $j_1$, $\orr(w)_k=\orr(w')_k$ whenever $k> j_1$.
		
		To show $\orr(w)_{j_1} = \orr(w')_{j_1}$, we argue by contradiction. Clearly $\orr(w)_{j_1}\leq \orr(w')_{j_1}$, so necessarily $\orr(w)_{j_1}<\orr(w')_{j_1}$. Since $i_1$ and $j_1$ occur consecutively in an element of $\LIS(w,i_1)$, it must be that $\orr(w)_{i_1}=\orr(w)_{j_1}+1$. Thus, $\orr(w)_{i_1}\leq \orr(w')_{j_1}$.
		
		On the other hand, take $\alpha\in\LIS(w',j_1)$. Replacing $\alpha_1=j_1$ by $i_1$ in $\alpha$ yields an increasing sequence $\beta$ in $w$ that starts from $i_1$ and has length $\orr(w')_{j_1}+1$. Hence, $\orr(w)_{i_1}+1\geq \orr(w')_{j_1}+1$. Thus we have $\orr(w)_{i_1}=\orr(w')_{j_1}$. Consequently $\beta\in \LIS(w,i_1)$. However $\beta_2=\alpha_2>j_1$, which contradicts the choice of $(i_1,j_1)$ as lexicographically last among $\LIS(w,i_1)$. This contradiction shows $\orr(w)_{j_1}=\orr(w')_{j_1}$.
		
		Now, fix any $k$ such that $i_1<k<j_1$. As $w\lessdot w'$, it follows that either $w(k)<w(i_1)$ or $w(k)>w(j_1)$. Take $\alpha\in\LIS(w,k)$. If $w(k)>w(j_1)$ then $\alpha$ does not include $j_1$, so $\alpha\in\LIS(w',k)$ as well. In this case, $\orr(w)_k=\orr(w')_k$.
		
		Otherwise, suppose $w(k)<w(i_1)$. Clearly, $\orr(w)_k\leq \orr(w')_k$. To reach a contradiction, suppose $\orr(w)_k<\orr(w')_k$. It follows that every $\alpha\in\LIS(w',k)$ includes $j_1$. Fix any $\alpha\in\LIS(w',k)$, and suppose $\alpha_p=j_1$. Since $\orr(w)_{j_1} = \orr(w')_{j_1}$, we can find an element $\beta\in \LIS(w,j_1)\cap \LIS(w',j_1)$. Construct a sequence $\gamma$ by letting $\gamma_r=\alpha_r$ for $r\in [p]$, and $\gamma_r=\beta_r$ for $r> p$. It follows that $\gamma$ has the same length as $\alpha$ and that $\gamma$ is an increasing subsequence of $w$. Consequently, $\orr(w')_k+1\leq \orr(w)_k+1$, a contradiction to our assumption that $\orr(w)_k<\orr(w')_k$. Hence, $\orr(w)_k=\orr(w')_k$.
		
		We now address the remaining case $k=i_1$. Suppose first that $\ell(\CN(w))>1$ and $(i_2,j_2)\notin M(\CN(w))$. It follows that $i_1=i_2$ and $j_1>j_2$. Then from the construction of $\CN(w)$, we see $i_1<j_2<j_1$ with $w(i_1)<w(j_1)<w(j_2)$. 
		
		Clearly $\orr(w)_{i_1}\geq \orr(w')_{i_1}$. Take $\alpha \in \LIS(w,i_1)$, with $\alpha_2=j_1$. Let $\beta$ be $\alpha$ with $\alpha_2=j_1$ dropped. Then $\beta$ is an increasing sequence in $w'$, so $\orr(w')_{i_1}\geq \orr(w)_{i_1}-1$. Thus, $\orr(w)_{i_1}\geq \orr(w')_{i_1}\geq \orr(w)_{i_1}-1$. To reach a contradiction, suppose that $\orr(w')_{i_1}= \orr(w)_{i_1}-1$. Then it must be that $\beta\in \LIS(w',i_1)$. The choice of $j_2$ requires that $\beta_2\leq j_2$. However, $\beta_2=\alpha_3>j_1$. Thus, $j_2>j_1$, a contradiction. Hence $\orr(w)_{i_1} = \orr(w')_{i_1}$.
		
		It remains to show that whenever $\ell(\CN(w))=1$ or $(i_2,j_2)\in M(\CN(w))$, we have $\orr(w)_{i_1} = \orr(w')_{i_1}+1$. When $\ell(\CN(w))=1$, we have $w'=w_0$. In this case, it is easy to see that $\orr(w)_{i_1}=1$ as needed. Suppose then that $\ell(\CN(w))>1$ and $(i_2,j_2)\in M(\CN(w))$. This implies that either $i_1<i_2$, or $i_1=i_2$ with $j_1<j_2$.
		
		Assume first that $i_1=i_2$ and $j_1<j_2$. From the construction of $\CN(w)$, we see $i_1<j_1<j_2$ with $w(i_1)<w(j_1)<w(j_2)$. 
		The choice of $j_1$ asserts that exists $\alpha\in \LIS(w,i_1)$ with $\alpha_2=j_1$. Hence $\orr(w)_{i_1}\leq \orr(w')_{i_1}+1$. We prove the reverse inequality. The choice of $j_2$ asserts that exists $\beta\in \LIS(w',i_1)$ with $\beta_2=j_2$. Inserting $j_1$ into $\beta$ produces an increasing subsequence of $w$ starting from $i_1$ of length $\orr(w')_{i_1}+2$, so $\orr(w)_{i_1}+1\geq \orr(w')_{i_1}+2$. This concludes the case $i_1=i_2$ and $j_1<j_2$.
		
		Now, assume $i_1<i_2$. The construction of $\CN(w)$ then implies that 
		\[w(j_1) = \max\{w(i_1),w(j_1+1),\ldots, w(n) \}.\] 
		Since there must be some $\alpha\in \LIS(w,i_1)$ that includes $j_1$, it follows that $\alpha = (i_1,j_1)$ is the only such sequence. Hence $\orr(w)_{i_1} = 1$ and $\orr(w')_{i_1} =0$ as needed.
	\end{proof}
	
	\begin{theorem}
		\label{thm:strongequality}
		Fix any $w\in S_n$ and let $\xi(w)=(\CN(w),M(\CN(w)))$. Then \[\dwt(\xi(w)) = \orr(w).\]
	\end{theorem}
	\begin{proof}
		We work by induction on $m=\ell(w_0)-\ell(w)$. If $m=0$, then $w=w_0$. Then $\CN(w_0)$ is empty, and both $\orr(w_0)$ and $\dwt(\xi(w))$ are the zero vector.
		
		Assume the result holds for all $m'<m$. Suppose $\ell(w)=\ell(w_0)-m$. Let $\CN(w)_p=(i_p,j_p)$ for $p\in [\ell(\CN(w))]$ and $w'=wt_{i_1j_1}$, so $\xi(w')=(\CN(w'),M(\CN(w')))$. By induction, $\dwt(\xi(w')) = \orr(w')$. The theorem now follows immediately from Lemma \ref{lem:markingrecursion} and Theorem \ref{thm:drecursion}.
	\end{proof}
	
	\begin{corollary}
		\label{cor:weakequality}		
		For any $w\in S_n$, $\#M(\CN(w)) = \absoverrajcode{w}$.
	\end{corollary}

	\begin{corollary} 
		\label{cor:deg>=}
		The degree of $\mathfrak{G}_w$ is at least $\absrajcode{w}$.
	\end{corollary}
	\begin{proof}
		By Theorem \ref{thm:strongequality}, the monomial $\bm{x}^{\rr(w)}$ has degree $\absrajcode{w}$ and lies in the support of $\mathfrak{G}_w$. Thus, $\deg\mathfrak{G}_w\geq \absrajcode{w}$. 
	\end{proof}

	\section{The degree of $\mathfrak{G}_w$ equals $\absrajcode{w}$}
	\label{sec:raj}

	We showed in Corollary \ref{cor:deg>=} that $\deg\mathfrak{G}_w\geq \absrajcode{w}$. In Corollary \ref{cor:raj} we show that $\deg\mathfrak{G}_w= \absrajcode{w}$. This yields an alternate proof of part of Theorem \ref{thm:pswleadingterm}.
	\medskip

	Let $C$ be a climbing chain of $w\in S_n$. Suppose $C$ consists of
	\[C: w=w^{(0)} \xrightarrow{C_1=(i_1,j_1)} w^{(1)} \xrightarrow{C_2=(i_2,j_2)} w^{(2)}\xrightarrow{C_3=(i_3,j_3)}\cdots\xrightarrow{C_m=(i_m,j_m)} w^{(m)} = w_0. \]
	Using Algorithm \ref{alg:1} below, we associate to $C$ two sequences of sets: $\Psi_0,\Psi_1,\ldots,\Psi_m$ and $\Omega_0,\Omega_1,\ldots,\Omega_m$.

	\begin{algorithm}
		\caption{}
		\label{alg:1}
		\begin{algorithmic}
			\State input $w\in S_n$
			\State input a climbing chain $C$ of $w$ with components
			\[C: w=w^{(0)} \xrightarrow{C_1=(i_1,j_1)} w^{(1)} \xrightarrow{C_2=(i_2,j_2)} w^{(2)}\xrightarrow{C_3=(i_3,j_3)}\cdots\xrightarrow{C_m=(i_m,j_m)} w^{(m)} = w_0\]
			
			\\
			\For{$q\in[n-1]$}
				\State compute the lexicographically last element $\alpha^{q}=(\alpha^q_1,\ldots,\alpha^q_{k_q})$ of $\LIS(w,q)$ 
			\EndFor
			\State initialize $\Psi_0 = \bigcup_{q=1}^{n-1}\left\{(q,\alpha^q_2),\ldots,(q,\alpha^q_{k_q}) \right\}$
			\State initialize $\Omega_0=\emptyset$
			\\

			\For{$k=1,2,\ldots,m$}
				\If{$C_k\in \Psi_{k-1}$} 
					\State set $\Psi_k=\Psi_{k-1}\setminus\{C_k\}$
					\State set $\Omega_k = \Omega_{k-1}\cup\{C_k\}$
				\ElsIf{$w^{(k)}(a')>w^{(k)}(b')$ for some $(a',b')\in \Psi_{k-1}$}
					\State set $\Psi_k=\{(t_{i_kj_k}(a), b) \mid (a,b)\in \Psi_{k-1} \}$
					\State set $\Omega_k=\Omega_{k-1}$
				\Else
					\State set $\Psi_k=\Psi_{k-1}$
					\State set $\Omega_k = \Omega_{k-1}$
				\EndIf
			\EndFor\\
			\Return $\Psi_0,\ldots,\Psi_m$ and $\Omega_0,\ldots,\Omega_m$
		\end{algorithmic}
	\end{algorithm}
\pagebreak
	\begin{definition}
		Observe that Algorithm \ref{alg:1} builds $\Omega_k$ and $\Psi_k$ by performing exactly one of three available transformations on $\Omega_{k-1}$ and $\Psi_{k-1}$: the ``if'', ``else if'', and ``else'' blocks. We will (respectively) name these three operations \emph{transfer}, \emph{adjust}, and \emph{pass}.
	\end{definition}

	\begin{example}
		\label{exp:265143algorithmrun}
		Let $w=265143$. The Rothe diagram of $w$ is
		\begin{center}
			\begin{tikzpicture}[scale=.55]
				\draw (0,0)--(6,0)--(6,6)--(0,6)--(0,0);
				\draw[draw=red] (6,5.5) -- (1.5,5.5) node[red] {$\bullet$} -- (1.5,0);
				\draw[draw=red] (6,4.5) -- (5.5,4.5) node[red] {$\bullet$} -- (5.5,0);
				\draw[draw=red] (6,3.5) -- (4.5,3.5) node[red] {$\bullet$} -- (4.5,0);
				\draw[draw=red] (6,2.5) -- (0.5,2.5) node[red] {$\bullet$} -- (0.5,0);
				\draw[draw=red] (6,1.5) -- (3.5,1.5) node[red] {$\bullet$} -- (3.5,0);
				\draw[draw=red] (6,0.5) -- (2.5,0.5) node[red] {$\bullet$} -- (2.5,0);
				
				\filldraw[draw=black,fill=lightgray] (0,5)--(1,5)--(1,6)--(0,6)--(0,5);
				\filldraw[draw=black,fill=lightgray] (0,4)--(1,4)--(1,5)--(0,5)--(0,4);
				\filldraw[draw=black,fill=lightgray] (0,3)--(1,3)--(1,4)--(0,4)--(0,3);
				\filldraw[draw=black,fill=lightgray] (2,1)--(3,1)--(3,2)--(2,2)--(2,1);
				\filldraw[draw=black,fill=lightgray] (2,4)--(3,4)--(3,5)--(2,5)--(2,4);
				\filldraw[draw=black,fill=lightgray] (3,4)--(4,4)--(4,5)--(3,5)--(3,4);
				\filldraw[draw=black,fill=lightgray] (4,4)--(5,4)--(5,5)--(4,5)--(4,4);
				\filldraw[draw=black,fill=lightgray] (2,3)--(3,3)--(3,4)--(2,4)--(2,3);
				\filldraw[draw=black,fill=lightgray] (3,3)--(4,3)--(4,4)--(3,4)--(3,3);
				
				\draw (0,1)--(6,1);
				\draw (0,2)--(6,2);
				\draw (0,3)--(6,3);
				\draw (0,4)--(6,4);
				\draw (0,5)--(6,5);
				
				\draw (1,0)--(1,6);
				\draw (2,0)--(2,6);
				\draw (3,0)--(3,6);
				\draw (4,0)--(4,6);
				\draw (5,0)--(5,6);
				
				\node at (-1.5,2.5) {$D(w)=$};
				\node at (6.3,2.5) {.};
			\end{tikzpicture}
		\end{center}
		Consider the chain
		\begin{align*}
			C: 265143 {\xrightarrow{(1,3)}} 562143
					 {\xrightarrow{(1,2)}} 652143
					 {\xrightarrow{(3,6)}} 653142
					 {\xrightarrow{(3,5)}} 654132
					 {\xrightarrow{(4,5)}} 654312 
					 {\xrightarrow{(5,6)}} 654321.
		\end{align*}
		Running Algorithm \ref{alg:1}, we start with $\Omega_0=\emptyset$. For each $q\in[5]$, the sets $\LIS(w,q)$ are
		\begin{align*}
			\LIS(w,1)&=\{(1, 2), (1, 3), (1, 5), (1, 6)\}\\
			\LIS(w,2)&=\{(2)\}\\
			\LIS(w,3)&=\{(3)\}\\
			\LIS(w,4)&=\{(4,5),(4,6)\}\\
			\LIS(w,5)&=\{(5)\}
		\end{align*}
		The lexicographically last elements of $\LIS(w,q)$ for $q=1,2,3,4,5$ are
		\[(1,6),(2),(3),(4,6),(5), \quad\mbox{so}\quad \Psi_0=\{(1,6),(4,6) \}. \]
		The algorithm terminates after the six steps shown in Figure \ref{fig:265143algorithmrun}.
	\end{example}

	\begin{figure}[ht]
		\begin{center}
			\includegraphics[scale=1.25]{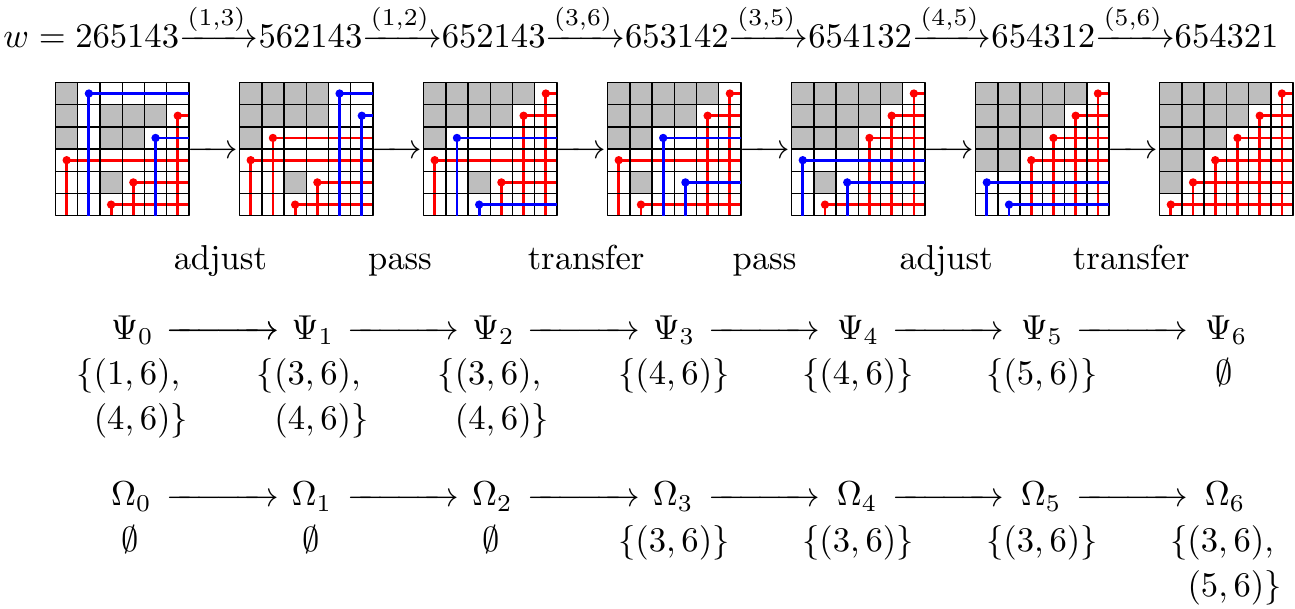}
		\end{center}
		\caption{Evaluation of Algorithm \ref{alg:1} on a climbing chain of $w=265143$.}
		\label{fig:265143algorithmrun}
	\end{figure}
	
	The following lemma describes a key property of the sets $\Psi_k$ from Algorithm \ref{alg:1}.

	\begin{lemma}
		\label{lem:algkeyproperty}
		Let $w\in S_n$ and $C$ be a climbing chain of $w$. Assume the notation of Algorithm \ref{alg:1}. Fix $k$ with $0\leq k\leq \ell(C)$. For each $q\in [n-1]$, let the elements of $\Psi_k$ of the form $(q,\star)$ be $(q,\beta^q_1),\ldots,(q,\beta^q_{l_q})$, labeled so $\beta^q_1<\cdots<\beta^q_{l_q}$. Then
		\[q<\beta^q_1, \quad\mbox{and}\quad w^{(k)}(q) < w^{(k)}(\beta^q_1) < \cdots < w^{(k)}(\beta^q_{l_q}).\]
	\end{lemma}
	\begin{proof}
		We work by finite induction on $k$. When $k=0$, the lemma follows from the definition of $\Psi_0$. Suppose the lemma holds for $k-1$. 
	
		Suppose first that $\Psi_k$ is obtained from $\Psi_{k-1}$ by a transfer operation. Then $\Psi_k=\Psi_{k-1}\setminus \{(i_k,j_k)\}$. Since $(i_k,j_k)$ is a link in $C$, this means it must be the lexicographically first element of $\Psi_{k-1}$ (by the induction assumption). The statement of the lemma follows easily in this case.
		
		Next, suppose that $\Psi_k$ is obtained from $\Psi_{k-1}$ by an adjust operation (See Figure \ref{fig:algbargraphsswap} for a visual representation of this case.). It follows $(i_k,j_k)\notin \Psi_{k-1}$, $w^{(k)}(a')>w^{(k)}(b')$ for some $(a',b')\in \Psi_{k-1}$, and \[\Psi_k = \{(t_{i_kj_k}(a),b)\mid (a,b)\in \Psi_{k-1} \}. \] Observe that all such $(a',b')$ must satisfy $a'=i_k$, $b'>j_k$, and $w^{(k-1)}(i_k)<w^{(k-1)}(b')<w^{(k-1)}(j_k)$. The adjust operation then guarantees the conditions of the lemma are met.
		
		Lastly, consider the case of a pass operation. This means that $(i_k,j_k)\notin \Psi_{k-1}$ and $w^{(k)}(a)<w^{(k)}(b)$ for all $(a,b)\in \Psi_{k-1}$. The conditions of the lemma are immediate.
	\end{proof}

	For later use, we separately record an observation made in the ``adjust'' case of the proof of Lemma \ref{lem:algkeyproperty}.

	\begin{lemma}
		\label{lem:adjustcase}
		Let $w\in S_n$ and $C$ be a climbing chain of $w$. Assume the notation of Algorithm \ref{alg:1}. Suppose step $k$ in the execution of Algorithm \ref{alg:1} is an adjust operation caused by $(a',b')\in \Psi_{k-1}$. Then $a'=i_k$ and $b'>j_k$.
	\end{lemma}

	\begin{figure}[ht]
		\begin{center}
			\includegraphics[scale=1]{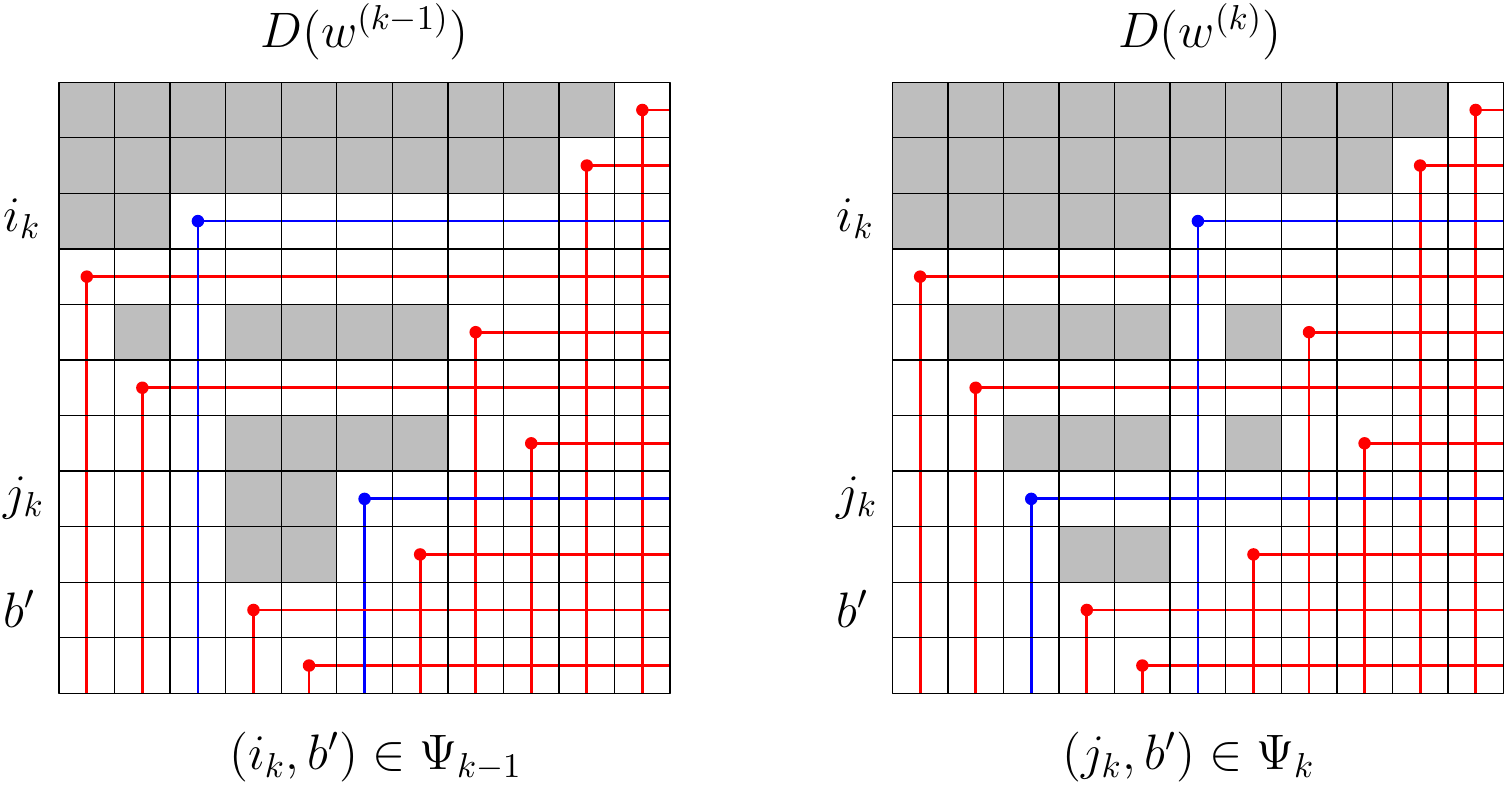}
		\end{center}
		\caption{Visual representations of the ``adjust'' case in Lemma \ref{lem:algkeyproperty}}.
		\label{fig:algbargraphsswap}
	\end{figure}

	The following lemma records basic properties of Algorithm \ref{alg:1}.
	\begin{lemma}
		\label{lem:algbasicproperties}
		Let $w\in S_n$ and $C$ be a climbing chain of $w$. Assume the notation of Algorithm \ref{alg:1}. Then 
		\begin{enumerate}[label=\textup{(\roman*)}]
			\item $\emptyset=\Omega_0\subseteq \cdots\subseteq \Omega_m\subseteq \{C_1,\ldots,C_m\}$.
			\item For each $0\leq k\leq m$, one has $\Omega_k\subseteq \{C_1,\ldots,C_k \}$.
			\item If $(i_p,j_p)\in \Omega_k$ for some $k$, then $\Omega_p=\Omega_{p-1}\cup \{(i_p,j_p)\}$ and $\Psi_p=\Psi_{p-1}\setminus \{(i_p,j_p)\}$.
			\item For each $0\leq k\leq m$, one has $\#\Omega_k + \#\Psi_k = \absoverrajcode{w}.$
			\item $\Psi_m=\emptyset$ and $\#\Omega_m=\absoverrajcode{w}$.
		\end{enumerate}
	\end{lemma}

	\begin{proof}
		We first address (i) and (ii). Initially, $\Omega_0=\emptyset$. Only the transfer operation causes $\Omega_k\neq \Omega_{k-1}$, specifically by adding the element $C_k=(i_k,j_k)$ to $\Omega_{k-1}$. Thus (i) and (ii) hold. For (iii), note that $(i_p,j_p)\in \Omega_k$ implies step $p$ of Algorithm \ref{alg:1} was a transfer operation. Claim (iv) holds by an argument analogous to that of (ii). The claim that $\Psi_m=\emptyset$ in (v) follows from Lemma \ref{lem:algkeyproperty}. That $\#\Omega_m=\absoverrajcode{w}$ then follows from (iv).
	\end{proof}

	\begin{definition}
		Let $C$ be a climbing chain of $w\in S_n$. Suppose the minimal markings $M(C)$ are 
		\[\{C_{k_1},C_{k_2},\ldots,C_{k_p}\}\quad \mbox{with}\quad 1\leq k_1<k_2<\cdots<k_p\leq \ell(C).\]
		Define the \emph{runs} of $C$ to be the (disjoint) subsequences of $C$ consisting of links $k_q$ through $k_{q+1}-1$ for each $q\in [p]$ (taking $k_{p+1}=\ell(C)+1$).
	\end{definition}
	
	\begin{example}
		\label{exp:265143links}
		Continuing Example \ref{exp:265143algorithmrun}, let $w=265143$ and $C=((1,3),(1,2),(3,6),(3,5),(4,5),(5,6))$. Indicating the minimal markings $M(C)$ by overlines and separating runs by ``$\mid$'', the runs of $C$ are
 		\[\overline{(1,3)},(1,2) \,\mid\, \overline{(3,6)},(3,5)\,\mid\, \overline{(4,5)} \,\mid\, \overline{(5,6)}.\qedhere\]
	\end{example}
	
	\begin{lemma}
		\label{lem:atmostoneperrun}
		Let $w\in S_n$ and $C$ be a climbing chain of $w$ with length $m$. Construct the sets $\Omega_0,\ldots,\Omega_m$ and $\Psi_0,\ldots,\Psi_m$ using Algorithm \ref{alg:1}. Then each run of $C$ contains at most one element of $\Omega_m$.
	\end{lemma}
	\begin{proof}
		Consider a run $C_{p},C_{p+1},\ldots,C_{q}$ of $C$ containing at least two elements of $\Omega_m$. The definition of run forces $\{C_p,C_{p+1},\ldots,C_q\}\cap M(C) = \{C_p\}$. It follows that $i_p=i_{p+1}=\cdots=i_{q}$ and $j_p>j_{p+1}>\cdots>j_q>i_p$. 
		
		First, suppose that $C_a,C_{a+1}\in \Omega_m$ for some $p\leq a< q$. By Lemma \ref{lem:algbasicproperties}(iii), $C_a,C_{a+1}\in \Psi_{a-1}$. But $i_a=i_{a+1}$ and $j_a>j_{a+1}$, so Lemma \ref{lem:algkeyproperty} implies
		\[w^{(a-1)}(i_a) < w^{(a-1)}(j_{a+1}) < w^{(a-1)}(j_a). \]
		This contradicts the climbing chain condition $w^{(a-1)}\lessdot w^{(a)} = w^{(a-1)}t_{i_aj_a}$ (see Lemma \ref{lem:cover}).

		Now suppose $C_a,C_b\in \Omega_m$ for some $a,b$ with $p\leq a<b-1$, $b\leq q$ and $b-a$ minimal. Additionally, suppose there are no adjust operations between steps $a$ and $b$ in the execution of Algorithm \ref{alg:1}. This implies $\Psi_a=\Psi_{a+1}=\cdots=\Psi_{b-1}$. Thus $C_a,C_b\in \Psi_{a-1}$, again contradicting the Bruhat cover condition on climbing chains.
		
		Lastly, suppose there was an adjust operation between steps $a$ and $b$ in the execution of Algorithm \ref{alg:1}. Say the first such adjust operation occurs at step $k$. Recall that $i_a=i_{a+1}=\cdots=i_k=\cdots=i_b$. Suppose first that there are no elements $(i_k,q)\in \Psi_{k}$. Then by Lemma \ref{lem:adjustcase}, no further adjust operations occur prior to step $b$. Thus there are no elements $(i_k,q)\in \Psi_{b-1}$. By Lemma \ref{lem:algbasicproperties}, this contradicts that $C_b\in \Omega_m$.
		
		Hence we may assume that there is an element of the form $(i_k,q)\in \Psi_k$. Then $(j_k,q)\in \Psi_{k-1}$, so Lemma \ref{lem:algkeyproperty} implies $q>j_k$. Hence $q>j_k>j_{k+1}>\cdots>j_b$, so $(i_k,q)\notin \{C_k,C_{k+1},\ldots,C_{b}\}$. 
		
		If there is only a single adjust operation between steps $a$ and $b$, this contradicts that $C_b\in \Omega_m$. If there is more than one adjust operation between steps $a$ and $b$, the same contradiction is reached by repeating the previous argument for each adjust operation. 
	\end{proof}

	\begin{figure}[ht]
		\begin{center}
			\includegraphics[scale=1]{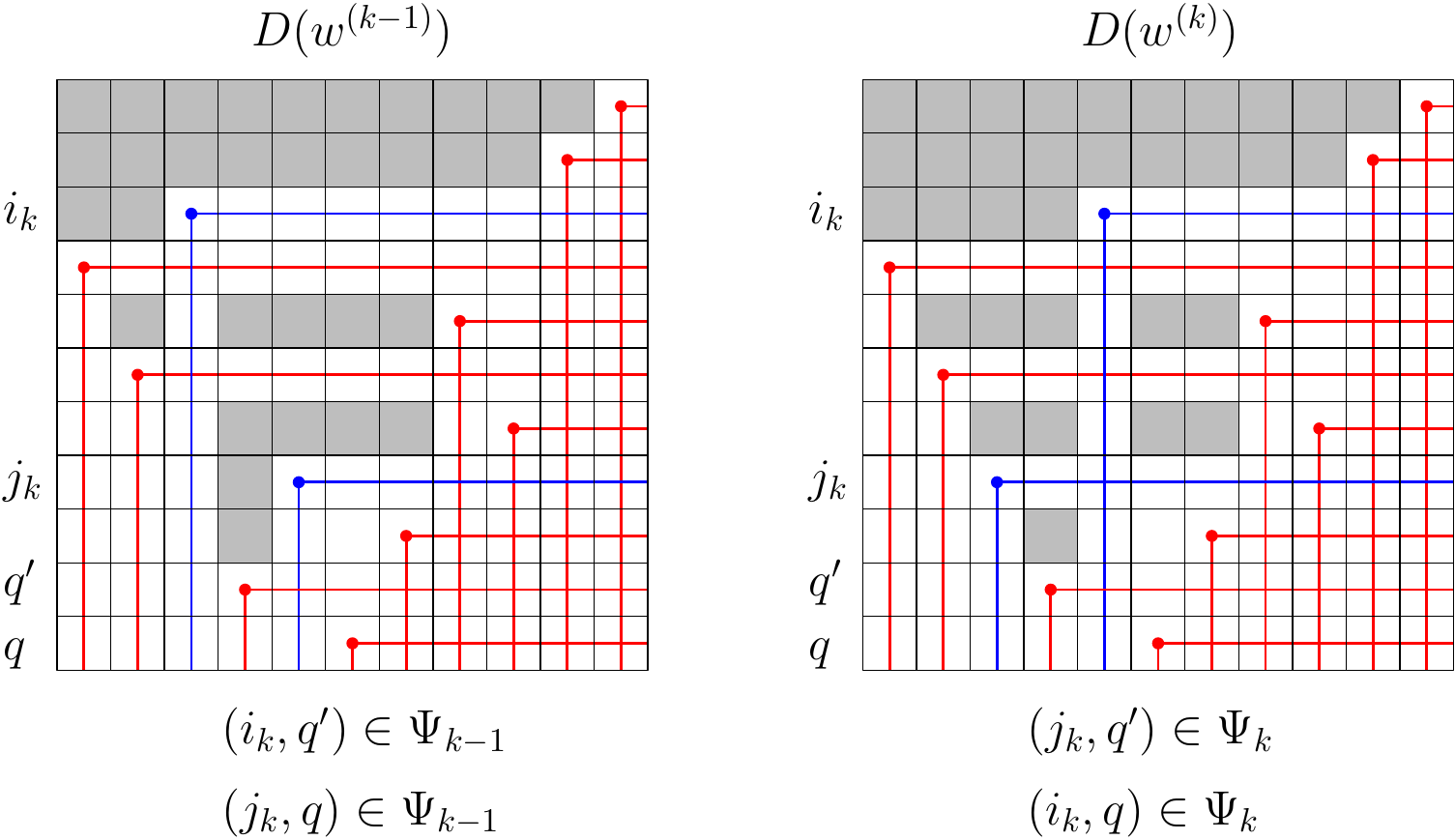}
			\caption{Visual aids for the ``adjust'' case in the proof of Lemma \ref{lem:atmostoneperrun}.}
			\label{fig:atmostonelemmavisual}
		\end{center}
	\end{figure}

	\begin{theorem}
		\label{prop:minnummarks}
		For any $w\in S_n$ and any climbing chain $C$ of $w$, 
		\[\absoverrajcode{w} \leq \#M(C). \]
	\end{theorem}
	\begin{proof}
		Use Algorithm \ref{alg:1} on $C$ to produce $\Omega_m$. Decompose $C$ into runs. By definition, the number of runs equals $\#M(C)$. Lemma \ref{lem:atmostoneperrun} shows there is at most one element of $\Omega_m$ in each run. Lemma \ref{lem:algbasicproperties}(v) implies $\#\Omega_m=\absoverrajcode{w}$. Putting this all together,
		\[\#M(C)\geq \#\Omega_m=\absoverrajcode{w}.\qedhere\]
	\end{proof}

	We are now ready to provide the alternative proof of the degree statement in Theorem \ref{thm:pswleadingterm}:

	\begin{corollary}[{\cite[Theorem 1.1]{PSW}}]
		\label{cor:raj} The degree of $\mathfrak{G}_w$ equals $\absrajcode{w}$.
	\end{corollary}
	\begin{proof}
		By Theorem \ref{thm:strongequality} and Corollary \ref{cor:deg>=}, $\absrajcode{w}$ is a lower bound on $\deg\mathfrak{G}_w$ that is attained by the nested chain. Theorem \ref{prop:minnummarks} implies no climbing chain of $w$ contributes a monomial of degree larger than $\absrajcode{w}$. Thus, $\deg \mathfrak{G}_w=\absrajcode{w}$.
	\end{proof}

	\section{Leading term of $\mathfrak{G}_w$ in term orders with $x_1<\cdots< x_n$}
	\label{sec:rajcode}
	We complete the alternative proof of Theorem \ref{thm:pswleadingterm}. In Corollary \ref{cor:raj} we showed the degree statement; below we show that $\bm{x}^{\rr(x)}$ is the leading monomial of $\mathfrak{G}_w^{\mathrm{top}}$ in any term order satisfying $x_1<\cdots< x_n$. 
	
	\begin{lemma}
		\label{lem:drecursionanychain}
		Let $C$ be a climbing chain of $w\in S_n$ with $w\neq w_0$. Suppose $C_p=(i_p,j_p)$ for $p\in [\ell(C)]$, and set $v=wt_{i_1j_1}$. Then 
		\begin{enumerate}[label=\textup{(\roman*)}]
			\item $\orr(w)_p=0=\orr(v)_p$ for $p\in [i_1-1]$,
			\item $\orr(w)_{i_1}\geq \orr(v)_{i_1}$,
			\item $\orr(w)_p\leq \orr(v)_p$ for $i_1+1\leq p \leq n$.
		\end{enumerate}
	\end{lemma}
	\begin{proof}
		Claim (i) is immediate since $i_1$ since $w(k)=v(k)=n-k+1$ for $k\in[i_1]$. Since $v(i_1)>w(i_1)$, any $\alpha\in \LIS(v,i_1)$ is also an increasing sequence in $w$. This shows (ii). 
		
		Consider claim (iii). When $p>j_1$, (iii) follows since $w$ and $v$ agree after position $j_1$. Since $v(j_1)<w(j_1)$, (iii) holds when $p=j_1$. It remains to consider $i_1<p<j_1$. Since $w\lessdot v$, either $w(p)<w(i_1)$ or $w(p)>w(j_1)$. In either case, any $\alpha\in \LIS(w,p)$ is also an increasing sequence in $v$. This proves (iii).
	\end{proof}
	
	\begin{definition}
		We call a climbing chain $C$ of $w\in S_n$ \emph{heavy} if $\#M(C) = \absoverrajcode{w}$.
	\end{definition}
	We choose the name ``heavy'' since such chains $C$ will have maximal total weight: $|\wt(C)| = \absrajcode{w}$.

	\begin{lemma}
		\label{lem:heavytruncation}
		Let $C$ be a climbing chain of $w\in S_n$, and assume the notation of Algorithm \ref{alg:1}.
		Suppose $C_k$ is the last link in $C$ of the form $(i_1,\star)$. Then 
		\begin{itemize}
			\item $\orr(w)_p=0=\orr(w^{(k)})_p$ for $p\in [i_1-1]$,
			\item $\orr(w)_{i_1}\geq \orr(w^{(k)})_{i_1}=0$,
			\item $\orr(w)_p\leq \orr(w^{(k)})_p$ for $i_1+1\leq p \leq n$.
		\end{itemize}
		whenever $i_1+1\leq p\leq n$. Additionally if $C$ is heavy, then the truncation $C'=(C_{k+1},\ldots,C_m)$ is heavy (as a climbing chain of $w^{(k)}$).
	\end{lemma}
	\begin{proof}
		The itemized claims follow from repeated applications of Lemma \ref{lem:drecursionanychain}. To see the final claim, suppose $C'$ is not heavy. Concatenating $(C_1,\ldots,C_{k})$ with a heavy chain for $(w^{(k)})$ yields a climbing chain of $w$ with fewer minimal marking than $C$, implying that $C$ is not heavy.
	\end{proof}

	\begin{lemma}
		\label{lem:heavymaxfirstmarks}
		Let $C$ be a heavy climbing chain of $w\in S_n$. Then
		\[\dwt(C,M(C))_{i_1}\leq \orr(w)_{i_1}. \]
	\end{lemma}
	\begin{proof}
		Assume the notation of Algorithm \ref{alg:1}. Suppose $k$ is the index of the last link in $C$ of the form $(i_1,\star)$ and $C'=(C_{k+1},\ldots,C_m)$.
		By the choice of $k$, we have
		\begin{itemize}
			\item $\dwt(C,M(C))_p=\dwt(C',M(C'))_p=0$ for $p<i_1$,
			\item $\dwt(C,M(C))_{i_1} >0$,
			\item $\dwt(C',M(C'))_{i_1} =0$,
			\item $\dwt(C,M(C))_p=\dwt(C',M(C'))_p$ for $p>i_1$.
		\end{itemize}
		Observe that the choice of $k$ implies $\orr(w^{(k)})_{p}=0$ for $p\leq i_1$. Applying Lemma \ref{lem:heavytruncation},
		\begin{align*}
			\left|\dwt(C,M(C))\right| &= \sum_{p=1}^{n}\dwt(C,M(C))_p \\
								 			 &= \dwt(C,M(C))_{i_1} + \sum_{p=i_1+1}^{n}\dwt(C,M(C))_{p} \\
								 			 &= \dwt(C,M(C))_{i_1} + \sum_{p=i_1+1}^{n}\dwt(C',M(C'))_{p} \\
								 			 &= \dwt(C,M(C))_{i_1} + \left|\dwt(C',M(C'))\right| \\
								 			 &= \dwt(C,M(C))_{i_1} +\absoverrajcode{w^{(k)}} \\
								 			 &= \dwt(C,M(C))_{i_1} +\orr(w^{(k)})_{i_1+1}+\cdots +\orr(w^{(k)})_n.
		\end{align*}
		On the other hand, $C$ being heavy implies
		\begin{align*}
			\left|\dwt(C,M(C))\right| &= \absoverrajcode{w}\\
											 &= \orr(w)_{i_1}+\cdots+\orr(w)_n.
		\end{align*}
		Hence,
		\begin{align}
			\label{eqn:1}
			\biggr(\dwt(C,M(C))_{i_1}-\orr(w)_{i_1}\biggr)+\sum_{p=i_1+1}^{n} \biggr(\orr(w^{(k)})_{p} -\orr(w)_{p}\biggr)=0.
		\end{align}
		By Lemma \ref{lem:drecursionanychain}, each term in the summation is nonnegative, so the leftmost term must be nonpositive. Thus,
		\[\dwt(C,M(C))_{i_1}\leq \orr(w)_{i_1}. \] 
	\end{proof}
		
	\begin{lemma}
		\label{lem:heavyequalitycase}
		Let $C$ be a heavy climbing chain of $w\in S_n$. Assume the notation of Algorithm \ref{alg:1}. 
		Suppose that $k$ is the index of the last link in $C$ of the form $(i_1,\star)$, and
		\[\dwt(C,M(C))_{i_1}= \orr(w)_{i_1}. \]
		Then
		\[
			\orr(w^{(k)})_p = \orr(w)_p \quad\mbox{for } p\geq i_1+1. 
		\]
	\end{lemma}
	\begin{proof}
		This is an immediate consequence of (\ref{eqn:1}).
	\end{proof}

	\begin{theorem}
		\label{thm:leadingchainweight}
		For $w\in S_n$ and any term order satisfying $x_1<x_2<\cdots<x_n$,
		\[\bm{x}^{\orr(w)} = \min \left\{\bm{x}^{\dwt(C,M(C))} \mid C \mbox{ is a heavy climbing chain of } w \right\}. \] 
		Equivalently,
		\[\bm{x}^{\rr(w)} = \max \left\{\bm{x}^{\wt(C,M(C))} \mid C \mbox{ is a heavy climbing chain of } w \right\}. \] 
	\end{theorem}
	\begin{proof}
		It is immediate from the definition of $\wt$ that the two assertions of the theorem are equivalent. We focus on the first. Since all heavy climbing chains have the same number of minimal markings, the theorem is equivalent to proving 
		\[\bm{x}^{\orr(w)} = \max \left\{\bm{x}^{\dwt(C,M(C))} \mid C \mbox{ is a heavy climbing chain of } w \right\}. \] 
		in any term order with $x_1>x_2>\cdots>x_n$. 
		For each $p\in[n]$, let $k_p$ be the index of the last link in $C$ of the form $(p,\star)$. The theorem follows by applying Lemmas \ref{lem:heavymaxfirstmarks} and \ref{lem:heavyequalitycase} (sequentially) to each of $w^{(k_1)},w^{(k_2)},\ldots,w^{(k_{n-1})}$.
	\end{proof}
	
	\noindent \textbf{Alternate Proof of 	Theorem \ref{thm:pswleadingterm}:} 
	Combine Theorem \ref{thm:leadingchainweight} and Theorem \ref{thm:grothformula}. \qed 
	
	\section{Interpolating Chains}
	\label{sec:inter}
	
	We define climbing chains that interpolate between the greedy and nested chains. We show that the monomials predicted by Hafner (Conjecture \ref{conj:elena}) to be the leading monomials of the homogeneous components of the Grothendieck polynomial $\mathfrak{G}_w$ in any term order satisfying $x_1<\cdots<x_n$ arise from these interpolating climbing chains. In the next section we show that they are indeed the leading monomials (confirming Conjecture \ref{conj:elena}).
	
	\begin{definition}
		Let $w\in S_n$ and $C$ be a climbing chain of $w$. 
		A link $C_p=(i_p,j_p)$ is called \emph{greedy} if $\greedy(wt_{i_1j_1}\cdots t_{i_{p-1}j_{p-1}}) = (i_p,j_p)$, and is called \emph{nested} if $\nested(wt_{i_1j_1}\cdots t_{i_{p-1}j_{p-1}}) = (i_p,j_p)$.
	\end{definition}
	Clearly $\CG(w)$ is the unique climbing chain of $w$ such that every link is greedy, and analogously for $\CN(w)$. Note that a given link can be both greedy and nested. We now define a family of chains that interpolate between $\CG(w)$ and $\CN(w)$.
	
	\begin{definition}
		\label{def:interpolatingchains}
		Fix $w\in S_n$ and let $m$ be the length of any climbing chain of $w$. We define the \emph{interpolating chains} of $w$ to be the climbing chains $I^0(w),\ldots,I^m(w)$ constructed as follows. 
		For $0\leq k\leq m$, define $I^k(w)$ to be the unique climbing chain of $w$ consisting of $m-k$ greedy links followed by $k$ nested links.
	\end{definition}
	
	Observe that $\CN(w) = I^m(w)$ and $\CG(w) = I^{0}(w) = I^{1}(w)$, since the final link in a chain must be both nested and greedy.
	
	\begin{example}
		\label{exp:5721463intchains}
		Let $w=5721463$. The distinct interpolating climbing chains of $w$ together with their underlying permutations are (with nested steps blue and greedy steps red)
		\begin{align*}
		I^0(w): 5721463 \textcolor{red}{\xrightarrow{(1,6)}} 6721453
		\textcolor{red}{\xrightarrow{(1,2)}} 7621453
		&\textcolor{red}{\xrightarrow{(3,7)}} 7631452
		\textcolor{red}{\xrightarrow{(3,5)}} 7641352
		\textcolor{red}{\xrightarrow{(3,6)}} 7651342\\ 
		\textcolor{red}{\xrightarrow{(4,7)}} 7652341
		&\textcolor{red}{\xrightarrow{(4,5)}} 7653241
		\textcolor{red}{\xrightarrow{(4,6)}} 7654231
		\textcolor{red}{\xrightarrow{(5,6)}} 7654321,\\
		\phantom{a}\\
		I^4(w): 5721463 \textcolor{red}{\xrightarrow{(1,6)}} 6721453
		\textcolor{red}{\xrightarrow{(1,2)}} 7621453
		&\textcolor{red}{\xrightarrow{(3,7)}} 7631452
		\textcolor{red}{\xrightarrow{(3,5)}} 7641352
		\textcolor{red}{\xrightarrow{(3,6)}} 7651342\\ 
		\textcolor{blue}{\xrightarrow{(4,5)}} 7653142
		&\textcolor{blue}{\xrightarrow{(4,6)}} 7654132
		\textcolor{blue}{\xrightarrow{(5,7)}} 7654231
		\textcolor{blue}{\xrightarrow{(5,6)}} 7654321,\\
		\phantom{a}\\
		I^7(w): 5721463 \textcolor{red}{\xrightarrow{(1,6)}} 6721453
		\textcolor{red}{\xrightarrow{(1,2)}} 7621453
		&\textcolor{blue}{\xrightarrow{(3,5)}} 7641253
		\textcolor{blue}{\xrightarrow{(3,6)}} 7651243
		\textcolor{blue}{\xrightarrow{(4,5)}} 7652143\\ 
		\textcolor{blue}{\xrightarrow{(4,7)}} 7653142
		&\textcolor{blue}{\xrightarrow{(4,6)}} 7654132
		\textcolor{blue}{\xrightarrow{(5,7)}} 7654231
		\textcolor{blue}{\xrightarrow{(5,6)}} ,7654321.
		\end{align*}
		The full collection of interpolating chains is 
		\begin{align*}
			I^0(w)&=I^1(w)=I^2(w)=I^3(w),\\
			I^4(w)&=I^5(w)=I^6(w),\\
			I^7(w)&=I^8(w)=I^9(w).
		\end{align*}
	\end{example}
	
	\begin{definition}
		To each permutation $w\in S_n$, we associate a sequence of vectors \[\leads(w)=(L_w(0),\ldots,L_w(d)),\] where $d=\absrajcode{w}-\absinvcode{w}$. Set $L_w(0)=\cc(w)$. Define $L_w(i)$ for $i\in [d]$ as follows: set $j\in [n]$ to be the largest index such that $L_w(i-1)_j<\rr(w)_j$, and define $L_w(i) = L_w(i-1) + e_j$.
	\end{definition}
	
	Recall that $\rr(w)_k\geq \cc(w)_k$ for all $k\in [n]$, so $L_w(d) = \rr(w)$. 
	
	\begin{example}
		\label{exp:5721463leads}
		Continuing Example \ref{exp:5721463intchains}, when $w=5721463$ we compute 
		\begin{align*}
		\cc(w) &= (4, 5, 1, 0, 1, 1, 0),\\
		\rr(w) &= (5, 5, 2, 1, 1, 1, 0).
		\end{align*}
		Thus $\leads(w) = (L_w(0), L_w(1), L_w(2), L_w(3))$, where
		\begin{align*}
		L_w(0) &= (4, 5, 1, 0, 1, 1, 0),\\
		L_w(1) &= (4, 5, 1, \textbf{1}, 1, 1, 0),\\
		L_w(2) &= (4, 5, \textbf{2}, 1, 1, 1, 0),\\
		L_w(3) &= (\textbf{5}, 5, 2, 1, 1, 1, 0).\qedhere
		\end{align*}
	\end{example}
	We now prove that the monomials $\bm{x}^{L_w(k)}$ appear with nonzero coefficient in $\mathfrak{G}_w$. We illustrate the key idea in the following example.
	
	\begin{example}
		\label{exp:5721463intchainsandleads}
		Continuing Examples \ref{exp:5721463intchains} and \ref{exp:5721463leads}, let $w=5721463$. We mark each interpolating chain $I^k(w)$ by
		\begin{itemize}
			\item completely marking the initial string of greedy links plus the first nested link, and
			\item minimally marking the remaining links.
		\end{itemize}
		Indicating markings with overlines and greedy/nested with red/blue respectively, we obtain the chains and exponents shown in Figure \ref{fig:markedchainexponents}.
	\end{example}

	\begin{figure}[ht]
		\begin{center}
			\includegraphics{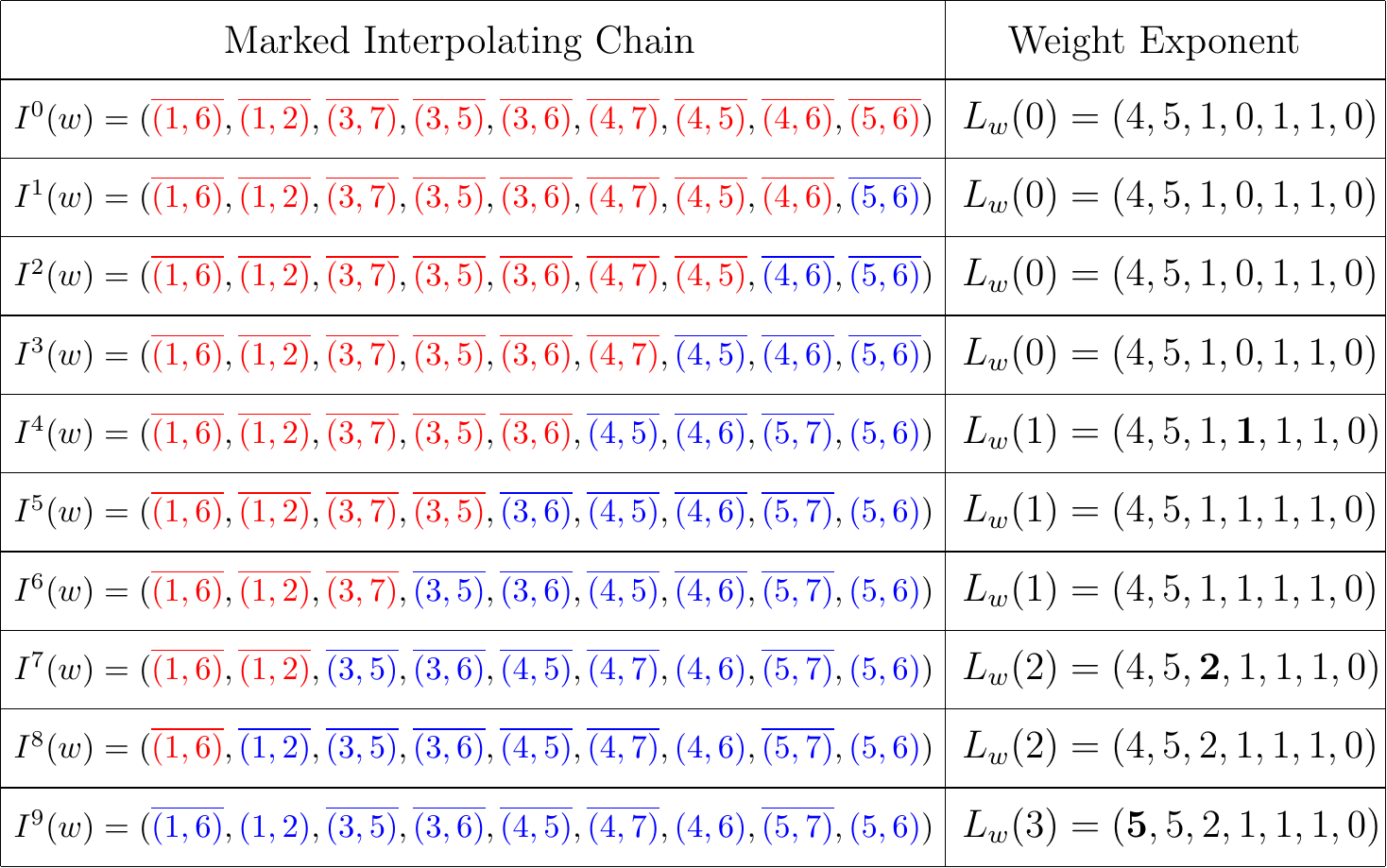}
		\end{center}
		\caption{}
		\label{fig:markedchainexponents}
	\end{figure}
	
	\begin{lemma}
		\label{lem:greedyrajchange}
		Fix $w\in S_n$. Let $\greedy(w)=(i,j)$ and $v=wt_{ij}$. Then either
		\[\orr(w) = \orr(v) \quad\mbox{or}\quad \orr(w) = \orr(v)+e_i.\]
		Equivalently, either $\rr(w) = \rr(v)$ or $\rr(w) = \rr(v)-e_i$.
	\end{lemma}
	\begin{proof}
		By definition, $i=\min\{k\mid w(k)\neq n+1-k\}$, and $j=w^{-1}(w(i)+1)$. Thus $v$ is obtained from $w$ by swapping the numbers $w(i)$ and $w(i)+1$. It is straightforward to check that $\orr(w)=\orr(v)$ unless every sequence $\alpha\in\LIS(w,i)$ has $\alpha_2=j$. In this case, $\orr(w)=\orr(v)+e_i$.
	\end{proof}
	
	\begin{lemma}
		\label{lem:greedyleadsjchange}
		Fix $w\in S_n$. Let $\greedy(w)=(i,j)$ and $v=wt_{ij}$. If $\leads(v)=(L_w(0),\ldots,L_w(d))$, then
		\[\leads(w) = 
		\begin{cases}
		(L_w(0)-e_i,\ldots,L_w(d)-e_i,\,\rr(w)) &\mbox{if } \rr(w) = \rr(v),\\
		(L_w(0)-e_i,\ldots,L_w(d)-e_i) &\mbox{if } \rr(w) = \rr(v)-e_i.
		\end{cases} 
		\] 
	\end{lemma}
	\begin{proof}
		The lemma follows from Lemma \ref{lem:greedyrajchange} and $\cc(w)=\cc(v)-e_i$.	
	\end{proof}
	
	\begin{theorem}
		\label{thm:interpolatingchainsexponents}
		Fix $w\in S_n$ and let $m$ be the length of any climbing chain of $w$. Let
		\[\leads(w)=(L_w(0),\ldots,L_w(d)).\] 
		For $0\leq p\leq m$, let $\xi_w(p)=(I^p(w),U^p(w))$ where $U^0(w) = I^0(w)$, and
		\[U^p(w) = M(I^p(w))\cup \{I^p(w)_1,I^p(w)_2,\ldots,I^p(w)_{m-p+1} \}\mbox{ for } p\in [m].\]
		Then 
		\[\left\{{L_w(0)},\ldots, {L_w(d)}\right\} = \left\{\wt({\xi_w(0)}),\ldots,\wt({\xi_w(m)}) \right\}. \]
	\end{theorem}
	\begin{proof}
		We work by induction on $\ell(w)$. When $w=w_0$ we have $d=m=0$ and $\xi_w(0) = (\emptyset,\emptyset)$. Thus
		\[\left\{ {L_w(0)}\right\} = \left\{(n-1,n-2,\ldots,1,0)\right\} = \left\{\wt({\xi(0)})\right\}. \] 
		
		Fix $w\in S_n$ and suppose the theorem holds for all permutations $v$ with $\ell(v)>\ell(w)$. Set $\greedy(w)=(i,j)$ and $v=wt_{ij}$. By definition, $\cc(w)=\cc(v)-e_i$. By Lemma \ref{lem:greedyrajchange}, either $\rr(w) = \rr(v)$ or $\rr(w) = \rr(v) - e_i$. 
		
		Suppose first that $\rr(w) = \rr(v)$. By Lemma \ref{lem:greedyleadsjchange}, $\leads(w)$ is obtained by subtracting $e_i$ from each element of $\leads(v)$, then appending $\rr(w)$.
		Thus,
		\[\left\{{L_w(0)},\ldots, {L_w(d)}\right\} = \left\{{L_v(0)}-e_i,\ldots, {L_v(d-1)}-e_i, \rr(w)\right\}.\]
		
		On the other hand, prepending and marking $(i,j)$ to all of the interpolating chains 
		$\xi_v(0),\ldots,\xi_v(m-1)$ of $v$ yields $\xi_w(0),\ldots, \xi_w(m-1)$. The final chain $\xi_w(m)$ is simply the nested chain of $w$ with its minimal markings. Thus
		\[\left\{\wt({\xi_w(0)}),\ldots,\wt({\xi_w(m)}) \right\} = \left\{\wt({\xi_v(0)})-e_i,\ldots,\wt({\xi_v(m-1)})-e_i,\rr(w) \right\}.\]
		Applying the induction assumption to $v$ completes the proof.
		
		The proof in the case that $\rr(w) = \rr(v) - e_i$ is almost identical. It is still true that
		\[\left\{\wt({\xi_w(0)}),\ldots,\wt({\xi_w(m)}) \right\} = \left\{\wt({\xi_v(0)})-e_i,\ldots,\wt({\xi_v(m-1)})-e_i,\rr(w) \right\}.\]
		However, Lemma \ref{lem:greedyleadsjchange} shows one instead obtains
		\[\left\{{L_w(0)},\ldots, {L_w(d)}\right\} = \left\{{L_v(0)}-e_i,\ldots, {L_v(d)}-e_i\right\}.\]
		In this case, one completes the proof by noting $\rr(w) =\rr(v)-e_i= \wt({\xi_v(m-1)})-e_i$.
	\end{proof}
	
	\begin{corollary} 
		\label{cor:exist} 
		The monomials appearing in Conjecture \ref{conj:elena} lie in the support of $\mathfrak{G}_w$.
	\end{corollary}

	\section{Proof of Conjecture \ref{conj:elena}}
	\label{sec:lead}
	
	In Corollary \ref{cor:exist} we showed that the monomials named in Conjecture \ref{conj:elena} are in the support of $\mathfrak{G}_w$. We show that these monomials are actually the leading monomials of the homogeneous components of $\mathfrak{G}_w$.

	Recall that when $\alpha\in\mathbb{R}^n$, $\overline{\alpha}$ is the vector with $\overline{\alpha}_k=n-k-\alpha_k$ for each $k\in[n]$.
	
	\begin{lemma}
		\label{lem:maxmarkedones}
		Let $(C,U)$ be any marked climbing chain of $w\in S_n$ and $C_1=(i_1,j_1)$. Then 
		\[\dwt(C,U)_{i_1} \leq \occ(w)_{i_1}.\]
		If equality holds, then 
		\[(C_1,C_2,\ldots,C_k) = (\CG_1,\ldots,\CG_k), \]
		where $C_k$ is the last link in $C$ of the form $(i_1,\star)$.
	\end{lemma}
	\begin{proof}
		From the definition of a climbing chain, we see that all links of the form $(i_1,\star)$ in $C$ are noninversions of $w$, pairs $(a,b)$ with $a<b$ and $w(a)<w(b)$. There are exactly $n-i_1-\cc(w)_{i_1} = \occ(w)_{i_1}$ of these noninversions with $a=i_1$. The second assertion follows from the nested hook interpretation of climbing chain links.
	\end{proof}

	\begin{lemma}
		\label{lem:onemarkbasecase}
		Let $(C,U)$ be a marked climbing chain of $w\in S_n$ with $C$ given by
		\[C: w=w^{(0)} \xrightarrow{C_1=(i_1,j_1)} w^{(1)} \xrightarrow{C_2=(i_2,j_2)} w^{(2)}\xrightarrow{C_3=(i_3,j_3)}\cdots\xrightarrow{C_m=(i_m,j_m)} w^{(m)} = w_0.\]
		Then
		\[\dwt(C,U)_{i_1} \leq \LL_w(m-\#U)_{i_1}. \]
	\end{lemma}
	\begin{proof}
		Suppose $C_k$ is the last link in $C$ of the form $(i_1,\star)$. 
		By Theorem \ref{prop:minnummarks},
		\[\#U = |\dwt(C,U)| = \dwt(C,U)_{i_1}+\sum_{p=i_1+1}^n\dwt(C,U)_p \geq \dwt(C,U)_{i_1} + \absoverrajcode{w^{(k)}}. \]
		On the other hand, Lemma \ref{lem:heavytruncation} implies 
		\[\absoverrajcode{w} \leq \orr(w)_{i_1} + \absoverrajcode{w^{(k)}},\]
		so that
		\[\#U = \absoverrajcode{w} + (\#U - \absoverrajcode{w})\leq \orr(w)_{i_1} + \absoverrajcode{w^{(k)}} + (\#U - \absoverrajcode{w}).\]
		Hence
		\[\dwt(C,U)_{i_1} \leq \orr(w)_{i_1} + (\#U - \absoverrajcode{w}).\]
		
		To conclude the proof, we will analyze $\orr(w)_{i_1} + (\#U - \absoverrajcode{w})$.
		Let $d=m-\#U$. See Figure \ref{fig:directions} for a visual representation of the following argument. Recall that $L_w(d)$ is obtained from $\cc(w)$ by iteratively increasing components $d$ times, moving right-to-left so entries stay weakly below the corresponding entries of $\rr(w)$. Then $\LL_w(d)$ is obtained from $\occ(w)$ by iteratively decreasing components $d$ times, moving right-to-left so entries stay weakly above the corresponding entries of $\orr(w)$.
		
		Equivalently, $\LL_w(d)$ is obtained from $\orr(w)$ by iteratively increasing components $m-\absoverrajcode{w}-d$ times, moving left-to-right so entries stay weakly below the corresponding entries of $\occ(w)$. Note that 
		\[m-\absoverrajcode{w}-d = \#U-\absoverrajcode{w}.\]
		If $\orr(w)_{i_1} + (\#U - \absoverrajcode{w})\geq \occ(w)_{i_1}$, then $\LL_w(d)_{i_1}=\occ(w)_{i_1}$ and Lemma \ref{lem:maxmarkedones} completes the proof. 
		
		Otherwise, $\orr(w)_{i_1} + (\#U - \absoverrajcode{w})< \occ(w)_{i_1}$, so 
		\[\LL_w(m-\#U)_{i_1} = \orr(w)_{i_1} + (\#U - \absoverrajcode{w}) \geq \dwt(C,U)_{i_1}\]
		as shown above, so we are done.
	\end{proof}

	\begin{figure}[ht]
		\begin{center}
			\includegraphics[scale=.8]{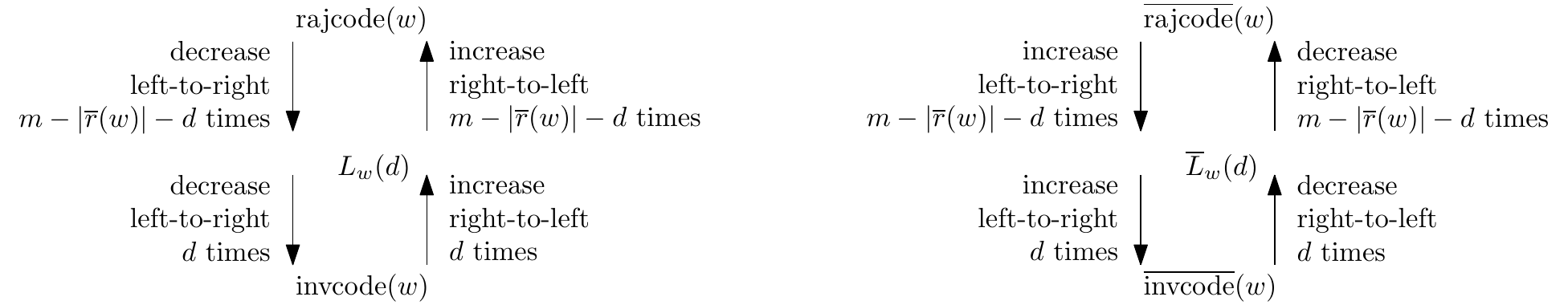}
			\caption{}
			\label{fig:directions}
		\end{center}
	\end{figure}

	\begin{theorem}
		\label{thm:interpolatingleadingterms}
		Fix $w\in S_n$. For each $d$ with $1\leq d\leq \absrajcode{w}-\absinvcode{w}$ and any term order satisfying $x_1<x_2<\cdots<x_n$,
		\[\bm{x}^{\LL_w(d)} = \min \left\{\bm{x}^{\dwt(C,U)} \mid (C,U) \mbox{ is a marked climbing chain of } w \mbox{ with } \#U = 
		\ell(C) - d \right\}. \] 
		Equivalently, 
		\[\bm{x}^{L_w(d)} = \max \left\{\bm{x}^{\wt(C,U)} \mid (C,U) \mbox{ is a marked climbing chain of } w \mbox{ with } \#U = 
		\ell(C) - d \right\}. \] 
	\end{theorem}
	\begin{proof}
		It is immediate from the definition of $\wt$ that the two assertions of the theorem are equivalent. We focus on the first. Since we are considering marked climbing chains with a fixed number of marks, the theorem is equivalent to proving 
		\[\bm{x}^{\LL_w(d)} = \max \left\{\bm{x}^{\dwt(C,M(C))} \mid (C,U) \mbox{ is a marked climbing chain of } w \mbox{ with } \#U = 
		\ell(C) - d \right\}. \] 
		in any term order with $x_1>x_2>\cdots>x_n$. 
		
		Assume the notation of Algorithm \ref{alg:1}. 
		By definition, $\LL_w(d)_{p}=\dwt(C,U)_{p}$ for $p<i_1$. Lemma \ref{lem:onemarkbasecase} shows that $\LL_w(d)_{i_1}\geq\dwt(C,U)_{i_1}$. If the inequality is strict, then there is nothing to prove. 
		
		Suppose first that $\LL_w(d)_{i_1}=\dwt(C,U)_{i_1}<\occ(w)_{i_1}$. Then for $p>i_1$, $\LL_w(d)_p=\orr(w)_p$. 
		Let $C_k$ be the last link in $C$ of the form $(i_1,\star)$. Let $C'=(C_{k+1},\ldots,C_m)$ and $U'=U\cap C'$.
		
		By Theorem \ref{prop:minnummarks},
		\[\sum_{p=i_1+1}^n \dwt(C,U)_p = |\dwt(C',U')|\geq \absoverrajcode{w^{(k)}}. \]
		For $p>i_1$, Lemma \ref{lem:heavytruncation} shows $\orr(w^{(k)})_p\geq \orr(w)_p$. Then
		\begin{align}
			\label{eqn:2}
			\sum_{p=i_1+1}^n \dwt(C,U)_p \geq \absoverrajcode{w^{(k)}}\geq \sum_{p=i_1+1}^n\orr(w)_p = \sum_{p=i_1+1}^n\LL_w(d)_p.
		\end{align}
		
		Since $|\LL_w(d)| = |\dwt(C,U)|$ and $\LL_w(d)_p = \dwt(C,U)_p$ for $p\leq i_1$, it follows that equality holds throughout in (\ref{eqn:2}). Thus $C'$ is heavy for $w^{(k)}$, so the theorem now follows from Theorem \ref{thm:leadingchainweight}.
			
		It remains to consider the case that $\LL_w(d)_{i_1}=\dwt(C,U)_{i_1}=\occ(w)_{i_1}$. In this case, we may pass to $w^{(k)}$ and repeat the above arguments and iterate. The iteration will terminate at worst with $w_0$, for which the theorem is trivial.
	\end{proof}

	\section{On leading monomials in term orders satisfying $x_1>\cdots>x_n$}
	\label{sec:>}
	We consider the leading monomials of the homogeneous components of $\mathfrak{G}_w$ in any term order satisfying $x_1>x_2>\cdots>x_n$. The lowest degree component, the Schubert polynomial $\mathfrak{S}_w$, is well-known to have leading term coming from the top-reduced pipe dream of the permutation \cite{laddermoves}. We construct this monomial from the leaping chain described below.
	
	Next, we construct the staircase chain, which we show is heavy and conjecture yields the leading monomial of $\mathfrak{G}_w^{\mathrm{top}}$ in any term order satisfying $x_1>x_2>\cdots>x_n$. We conclude with an analogue of Conjecture \ref{conj:elena}.

	\subsection{Leaping Chain}
	\label{sec:leap}
	\begin{definition}
		To each permutation $w\in S_n$ we associate a pair of integers $\leap(w)$, called the \emph{leap pair} of $w$, as follows. The leap pair of $w_0$ is undefined. If $w\neq w_0$, set 
		\[a=\min\{k\mid w(k)\neq n+1-k\}\quad \mbox{and}\quad b=\min\{k\mid k>a\mbox{ and }w\lessdot wt_{ak}\}.\] 
		Then $\leap(w)=(a,b)$.
	\end{definition}
	
	\begin{definition}
		Define the \emph{leaping climbing chain} $\CL(w)$ of $w\in S_n$ as follows. If $w=w_0$, then $\CL(w)$ is the empty sequence. Otherwise, let $\leap(w)=(a,b)$. Define $\CL(w)$ inductively by prepending $(a,b)$ to $\CL(wt_{ab})$.
	\end{definition}
	
	The following example demonstrates a graphical interpretation of $\leap(w)$.
	\begin{example}
		Continuing Example \ref{exp:31452marked} with $w=256341$, we compute $\CL(w)$.
		Recall the Rothe diagram of $w$ is
		\begin{center}
			\begin{tikzpicture}[scale=.55]
			\draw (0,0)--(6,0)--(6,6)--(0,6)--(0,0);
			\draw[draw=red] (6,5.5) -- (1.5,5.5) node[red] {$\bullet$} -- (1.5,0);
			\draw[draw=red] (6,4.5) -- (4.5,4.5) node[red] {$\bullet$} -- (4.5,0);
			\draw[draw=red] (6,3.5) -- (5.5,3.5) node[red] {$\bullet$} -- (5.5,0);
			\draw[draw=red] (6,2.5) -- (2.5,2.5) node[red] {$\bullet$} -- (2.5,0);
			\draw[draw=red] (6,1.5) -- (3.5,1.5) node[red] {$\bullet$} -- (3.5,0);
			\draw[draw=red] (6,0.5) -- (0.5,0.5) node[red] {$\bullet$} -- (0.5,0);
			
			\filldraw[draw=black,fill=lightgray] (0,5)--(1,5)--(1,6)--(0,6)--(0,5);
			\filldraw[draw=black,fill=lightgray] (0,4)--(1,4)--(1,5)--(0,5)--(0,4);
			\filldraw[draw=black,fill=lightgray] (0,3)--(1,3)--(1,4)--(0,4)--(0,3);
			\filldraw[draw=black,fill=lightgray] (0,2)--(1,2)--(1,3)--(0,3)--(0,2);
			\filldraw[draw=black,fill=lightgray] (0,1)--(1,1)--(1,2)--(0,2)--(0,1);
			\filldraw[draw=black,fill=lightgray] (2,4)--(3,4)--(3,5)--(2,5)--(2,4);
			\filldraw[draw=black,fill=lightgray] (3,4)--(4,4)--(4,5)--(3,5)--(3,4);
			\filldraw[draw=black,fill=lightgray] (2,3)--(3,3)--(3,4)--(2,4)--(2,3);
			\filldraw[draw=black,fill=lightgray] (3,3)--(4,3)--(4,4)--(3,4)--(3,3);
			
			\draw (0,1)--(6,1);
			\draw (0,2)--(6,2);
			\draw (0,3)--(6,3);
			\draw (0,4)--(6,4);
			\draw (0,5)--(6,5);
			
			\draw (1,0)--(1,6);
			\draw (2,0)--(2,6);
			\draw (3,0)--(3,6);
			\draw (4,0)--(4,6);
			\draw (5,0)--(5,6);
			
			\node at (-1.5,2.5) {$D(w)=$};
			\node at (6.3,2.5) {.};
			\end{tikzpicture}
		\end{center}
		Let $\leap(w)=(i_1,j_1)$. Since $w(1)\neq 6$, $i_1=1$. The definition of the leaping pair says that \[j_1=\min\{k\mid k>a\mbox{ and }w\lessdot wt_{ak}\}=2.\] Graphically, $j_1$ is the row index of the northmost dot inside the region south and east of the hook in row $i_1$.
		
		The Rothe diagram of $wt_{12}=526341$ is
		\begin{center}
			\begin{tikzpicture}[scale=.55]
			\draw (0,0)--(6,0)--(6,6)--(0,6)--(0,0);
			\draw[draw=red] (6,4.5) -- (1.5,4.5) node[red] {$\bullet$} -- (1.5,0);
			\draw[draw=red] (6,5.5) -- (4.5,5.5) node[red] {$\bullet$} -- (4.5,0);
			\draw[draw=red] (6,3.5) -- (5.5,3.5) node[red] {$\bullet$} -- (5.5,0);
			\draw[draw=red] (6,2.5) -- (2.5,2.5) node[red] {$\bullet$} -- (2.5,0);
			\draw[draw=red] (6,1.5) -- (3.5,1.5) node[red] {$\bullet$} -- (3.5,0);
			\draw[draw=red] (6,0.5) -- (0.5,0.5) node[red] {$\bullet$} -- (0.5,0);
			
			\filldraw[draw=black,fill=lightgray] (0,5)--(1,5)--(1,6)--(0,6)--(0,5);
			\filldraw[draw=black,fill=lightgray] (0,4)--(1,4)--(1,5)--(0,5)--(0,4);
			\filldraw[draw=black,fill=lightgray] (0,3)--(1,3)--(1,4)--(0,4)--(0,3);
			\filldraw[draw=black,fill=lightgray] (0,2)--(1,2)--(1,3)--(0,3)--(0,2);
			\filldraw[draw=black,fill=lightgray] (0,1)--(1,1)--(1,2)--(0,2)--(0,1);
			\filldraw[draw=black,fill=lightgray] (1,5)--(2,5)--(2,6)--(1,6)--(1,5);
			\filldraw[draw=black,fill=lightgray] (2,5)--(3,5)--(3,6)--(2,6)--(2,5);
			\filldraw[draw=black,fill=lightgray] (3,5)--(4,5)--(4,6)--(3,6)--(3,5);
			\filldraw[draw=black,fill=lightgray] (2,3)--(3,3)--(3,4)--(2,4)--(2,3);
			\filldraw[draw=black,fill=lightgray] (3,3)--(4,3)--(4,4)--(3,4)--(3,3);
			
			\draw (0,1)--(6,1);
			\draw (0,2)--(6,2);
			\draw (0,3)--(6,3);
			\draw (0,4)--(6,4);
			\draw (0,5)--(6,5);
			
			\draw (1,0)--(1,6);
			\draw (2,0)--(2,6);
			\draw (3,0)--(3,6);
			\draw (4,0)--(4,6);
			\draw (5,0)--(5,6);
			
			\node at (-2,2.5) {$D(wt_{12})=$};
			\node at (6.3,2.5) {.};
			\end{tikzpicture}
		\end{center}
		From the Rothe diagram, we observe that $\leap(wt_{12}) = (1,3)$. Continuing in this fashion yields 
		\[\CL(w)=((1,2),(1,3),(2,3),(3,4),(3,5),(4,5)). \]
		This is the chain $C^{(3)}$ in Example \ref{exp:31452unmarked}.
	\end{example}

	\begin{definition}
		Fix $w\in S_n$. Define the \emph{highest nesting length} $h(w)$ as follows.
		If $w=w_0$, set $h(w)=0$. Otherwise, set $(q_0,q_1)=\leap(w)$. Iteratively define 
		\[q_k=\min\{p\mid p>q_{k-1}\mbox{ and }wt_{q_0q_1}\cdots t_{q_{0}q_{k-1}}\lessdot wt_{q_0q_1}\cdots t_{q_{0}q_{k-1}} t_{q_{0}p}\}\] 
		until the set on the right side is empty. Let the $h(w)=\#\{q_1,q_2,\ldots\}$ be the number of steps taken.
	\end{definition}
	
	Graphically, the quantity $h(w)$ counts the sequence of highest nested hooks south and east of the dot in row $q_0$ of $D(w)$.
	\begin{example}
		For the $w=1465273$, $h(w)=3$ counts the hooks indicated in Figure \ref{fig:hwexample}.
	\end{example}

	\begin{figure}[ht]
		\begin{center}
			\includegraphics{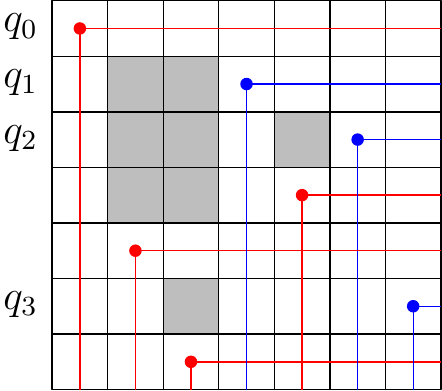}
		\end{center}
		\caption{}
		\label{fig:hwexample}
	\end{figure}

	\begin{lemma}
		\label{lem:hwminimal}
		For any $w\in S_n$ and marked climbing chain $(C,U)$ of $w$, 
		\[\dwt(C,U)_{i_1}\geq h(w). \]
	\end{lemma}
	\begin{proof}
		If $w=w_0$, the lemma reduces to $0\geq 0$. Otherwise, let $q=\min\{j\mid w(j)\neq n+1-j\}$. We work by induction. Let $C_p=(i_p,j_p)$ for all $p$, so $q=i_1$. Set $w'=wt_{i_1j_1}$ and $(C',U')$ be the corresponding truncation of $(C,U)$. Suppose $\leap(w)=(i_1,b)$ If $j_1\neq b$, then $j_1>b$, so $h(w')=h(w)$. In this case, induction implies
		\[\dwt(C,U)_{i_1} \geq \dwt(C',U')_{i_1} \geq h(w')=h(w). \] 
		Hence we may suppose that $j_1=b$. If $w'(i_1)=n-i_1+1$, then $\dwt(C,U)_{i_1}=1=h(w)$. Otherwise, $i_1=i_2$ and the definition of $\leap(w)$ forces $j_2>j_1$. Then $(i_2,j_2)\in U$, so by induction,
		\[\dwt(C,U)_{i_1} = \dwt(C',U')_{i_1}+1 \geq h(w')+1 = h(w). \]
	\end{proof}

	\begin{lemma}
		\label{lem:hwminimalsamestart}
		Fix any $w\in S_n$ and let $C$ be a climbing chain $C$ of $w$ with $\dwt(C,C)_{i_1}=h(w)$. Let $C_k$ be the last link in $C$ of the form $(i_1,\star)$. Then for each $p\in[k]$,
		\[C_p=\CL(w)_p.\]
	\end{lemma}
	\begin{proof}
		Suppose $C_1\neq\CL(w)_1$. Let $C_1=(i_1,j_1)$ and $\CL(w)_1=(i_1,b)$, so $b<j_1$. Then $h(wt_{i_1j_1}) = h(w) = h(wt_{i_1b})+1$. By Lemma \ref{lem:hwminimal}, it follows that
		\[\dwt(C,C)\geq h(w)+1>h(w)=\dwt(\CL(w),\CL(w)).\]
		This would contradict the assumption $\dwt(C,C)= h(w)$. Thus $C_1=\CL(w)_1$. Iterating this argument proves the lemma.
	\end{proof}

	\begin{theorem}
		\label{thm:leapingleading}
		In any term order with $x_1>x_2>\dots>x_n$, the leading monomial of the Schubert polynomial $\mathfrak{S}_w$ is $\bm{x}^{\wt(\CL(w),\CL(w))}$.
	\end{theorem}
	\begin{proof}
		The theorem follows from a straightforward induction using Lemmas \ref{lem:hwminimal} and \ref{lem:hwminimalsamestart}.
	\end{proof}

	\subsection{Staircase Chain}
	\label{sec:stair}
	\begin{definition}
		To each permutation $w\in S_n$ we associate a sequences of pairs $\stair(w)$, called the \emph{staircase} of $w$, as follows. The staircase of $w_0$ is undefined. If $w\neq w_0$, let
		$a=\min\{k\mid w(k)\neq n+1-k\}$ and consider $K=\{k\mid k>a\mbox{ and }w\lessdot wt_{ak}\}$.
		Let $K_1$ be the set of elements $k\in K$ with
		\[\orr(w)_k=\max(\{\orr(w)_{k'}\mid k'\in K \}).\]
		Iteratively, let $K_p$ be the set of elements $k\in K$ with $k<\min(K_{p-1})$ and
		\[\orr(w)_k=\max\left\{\orr(w)_{k'}\,\middle|\, k'\in K-\bigcup_{q=1}^{p-1}K_q \right\}.\]
		Suppose this process results in sets $K_1,K_2,\ldots,K_q$.		
		Then $\stair(w)$ is obtained by ordering the pairs
		\[\left\{(a,k)\mid k\in \bigcup_{p=1}^q K_p\right\} \]
		so their second components are decreasing.
	\end{definition}

	\begin{definition}
		Define the \emph{stairase climbing chain} $\CS(w)$ of $w\in S_n$ as follows. If $w=w_0$, then $\CS(w)$ is the empty sequence. Otherwise, let $\stair(w)=((a,b_1),(a,b_2),\ldots,(a,b_k))$. Define $\CS(w)$ inductively by concatenating $\stair(w)$ and $\CS(wt_{ab_1}\cdots t_{ab_k})$.
	\end{definition}

	\begin{example}
		Consider the permutation $w=1764352$. Below we draw the Rothe diagram of $w$ and label the hook in each relevant row with the corresponding value of $\orr(w)$, the length of the longest hook nesting below itself.
		\begin{center}
			\begin{tikzpicture}[scale=.55]
			\draw (0,0)--(7,0)--(7,7)--(0,7)--(0,0);
			\draw[draw=red] (7,6.5) -- (0.5,6.5) node[red] {$\bullet$} -- (0.5,0);
			\draw[draw=red] (7,5.5) -- (6.5,5.5) node[red] {$\bullet$} -- (6.5,0);
			\draw[draw=red] (7,4.5) -- (5.5,4.5) node[red] {$\bullet$} -- (5.5,0);
			\draw[draw=red] (7,3.5) -- (3.5,3.5) node[red] {$\bullet$} -- (3.5,0);
			\draw[draw=red] (7,2.5) -- (2.5,2.5) node[red] {$\bullet$} -- (2.5,0);
			\draw[draw=red] (7,1.5) -- (4.5,1.5) node[red] {$\bullet$} -- (4.5,0);
			\draw[draw=red] (7,0.5) -- (1.5,0.5) node[red] {$\bullet$} -- (1.5,0);
			
			\node[left] at (6.6, 5.5) {0};
			\node[left] at (5.6, 4.5) {0};
			\node[left] at (3.6, 3.5) {1};
			\node[left] at (2.6, 2.5) {1};
			\node[left] at (1.6, 0.5) {0};
			
			\filldraw[draw=black,fill=lightgray] (1,5)--(2,5)--(2,6)--(1,6)--(1,5);
			\filldraw[draw=black,fill=lightgray] (2,5)--(3,5)--(3,6)--(2,6)--(2,5);
			\filldraw[draw=black,fill=lightgray] (3,5)--(4,5)--(4,6)--(3,6)--(3,5);
			\filldraw[draw=black,fill=lightgray] (4,5)--(5,5)--(5,6)--(4,6)--(4,5);
			\filldraw[draw=black,fill=lightgray] (5,5)--(6,5)--(6,6)--(5,6)--(5,5);
			
			\filldraw[draw=black,fill=lightgray] (1,4)--(2,4)--(2,5)--(1,5)--(1,4);
			\filldraw[draw=black,fill=lightgray] (2,4)--(3,4)--(3,5)--(2,5)--(2,4);
			\filldraw[draw=black,fill=lightgray] (3,4)--(4,4)--(4,5)--(3,5)--(3,4);
			\filldraw[draw=black,fill=lightgray] (4,4)--(5,4)--(5,5)--(4,5)--(4,4);
			
			\filldraw[draw=black,fill=lightgray] (1,3)--(2,3)--(2,4)--(1,4)--(1,3);
			\filldraw[draw=black,fill=lightgray] (2,3)--(3,3)--(3,4)--(2,4)--(2,3);
			
			\filldraw[draw=black,fill=lightgray] (1,2)--(2,2)--(2,3)--(1,3)--(1,2);
			
			\filldraw[draw=black,fill=lightgray] (1,1)--(2,1)--(2,2)--(1,2)--(1,1);
			
			\draw (0,1)--(7,1);
			\draw (0,2)--(7,2);
			\draw (0,3)--(7,3);
			\draw (0,4)--(7,4);
			\draw (0,5)--(7,5);
			\draw (0,6)--(7,6);
			
			\draw (1,0)--(1,7);
			\draw (2,0)--(2,7);
			\draw (3,0)--(3,7);
			\draw (4,0)--(4,7);
			\draw (5,0)--(5,7);
			\draw (6,0)--(6,7);
			
			\node at (-1.5,2.5) {$D(w)=$};
			\node at (6.3,2.5) {.};
			\end{tikzpicture}
		\end{center}
	We have $K=\{7,5,4,3,2\}$, with $K_1=\{5,4\}$ and $K_2=\{3,2\}$. Thus,
	\[\stair(w) = ((1,5),(1,4),(1,3),(1,2)). \]
	Continuing in this fashion with $wt_{15}t_{14}t_{13}t_{12}=7643152$, one obtains
	\[\CS(w) = ((1,5),(1,4),(1,3),(1,2),(3,6),(4,6),(5,7),(5,6)).\qedhere\]
	\end{example}
	
	\begin{theorem}
		\label{thm:staircaseheavy}
		For any $w\in S_n$, $\CS(w)$ is a heavy climbing chain.	
	\end{theorem}
	\begin{proof}
		The verification that $\CS(w)$ is a climbing chain of $w$ is straightforward. 
		Recall that $\CS(w)$ is the concatenation of $\stair(w)$ and another staircase chain $\CS(w')$. 
		Observe that the sequence $\stair(w)$ will contribute only its first element to $M(\CS(w))$ by definition. If we can show that $\absoverrajcode{w}= \absoverrajcode{w'}+1$, then we are done by induction. 
		
		We show that $\absoverrajcode{w}= \absoverrajcode{w'}+1$ by tracking how $\absoverrajcode{\cdot}$ changes as the transpositions corresponding to the links in $\stair(w)$ are applied iteratively to $w$. Recall the sets $K_1,\ldots,K_q$ used to define $\stair(w)$. 
		
		To see the following claims, refer to the visual aid in Figure \ref{fig:staircase}. Each link from $K_1$ other than the last such link does not change the previous value $\absoverrajcode{\cdot}$. The last link from $K_1$ decreases the previous value $\absoverrajcode{\cdot}$ by 1. For links from each subsequent set $K_p$, 
		\begin{itemize}
			\item the first such link increases $\absoverrajcode{\cdot}$ by 1;
			\item intermediate such links do not change $\absoverrajcode{\cdot}$;
			\item the final such link decreases $\absoverrajcode{\cdot}$ by 1.
		\end{itemize}
	Any singleton $K_p$ acts as both the first and final link, leaving $\absoverrajcode{\cdot}$ unchanged. Thus, \[\absoverrajcode{w}= \absoverrajcode{w'}+1.\qedhere\]
	\end{proof}	

	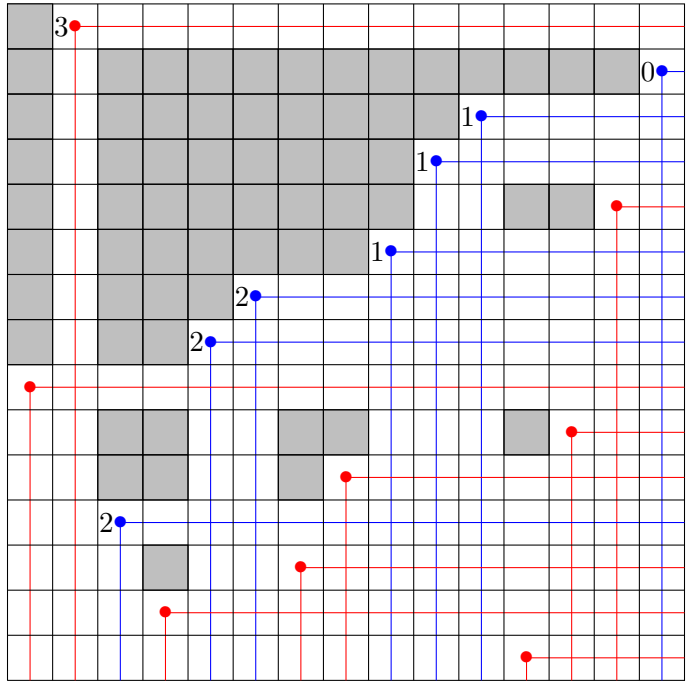
\begin{figure}[ht]
		\begin{center}
			\begin{tikzpicture}[scale=.6]
				\draw (0,0)--(15,0)--(15,15)--(0,15)--(0,0);
				
				\draw[draw=red] (15,14.5) -- (1.5,14.5) node[red] {$\bullet$} -- (1.5,0); 
				\draw[draw=blue] (15,13.5) -- (14.5,13.5) node[blue] {$\bullet$} -- (14.5,0); 
				\draw[draw=blue] (15,12.5) -- (10.5,12.5) node[blue] {$\bullet$} -- (10.5,0); 
				\draw[draw=blue] (15,11.5) -- (9.5,11.5) node[blue] {$\bullet$} -- (9.5,0); 
				\draw[draw=red] (15,10.5) -- (13.5,10.5) node[red] {$\bullet$} -- (13.5,0); 
				\draw[draw=blue] (15,9.5) -- (8.5,9.5) node[blue] {$\bullet$} -- (8.5,0); 
				\draw[draw=blue] (15,8.5) -- (5.5,8.5) node[blue] {$\bullet$} -- (5.5,0); 
				\draw[draw=blue] (15,7.5) -- (4.5,7.5) node[blue] {$\bullet$} -- (4.5,0); 
				\draw[draw=red] (15,6.5) -- (0.5,6.5) node[red] {$\bullet$} -- (0.5,0); 
				\draw[draw=red] (15,5.5) -- (12.5,5.5) node[red] {$\bullet$} -- (12.5,0); 
				\draw[draw=red] (15,4.5) -- (7.5,4.5) node[red] {$\bullet$} -- (7.5,0); 
				\draw[draw=blue] (15,3.5) -- (2.5,3.5) node[blue] {$\bullet$} -- (2.5,0); 
				\draw[draw=red] (15,2.5) -- (6.5,2.5) node[red] {$\bullet$} -- (6.5,0); 
				\draw[draw=red] (15,1.5) -- (3.5,1.5) node[red] {$\bullet$} -- (3.5,0); 
				\draw[draw=red] (15,0.5) -- (11.5,0.5) node[red] {$\bullet$} -- (11.5,0); 
				
				\filldraw[draw=black,fill=lightgray] (0,14)--(1,14)--(1,15)--(0,15)--(0,14); 
				\filldraw[draw=black,fill=lightgray] (0,13)--(1,13)--(1,14)--(0,14)--(0,13); 
				\filldraw[draw=black,fill=lightgray] (2,13)--(3,13)--(3,14)--(2,14)--(2,13); 
				\filldraw[draw=black,fill=lightgray] (3,13)--(4,13)--(4,14)--(3,14)--(3,13); 
				\filldraw[draw=black,fill=lightgray] (4,13)--(5,13)--(5,14)--(4,14)--(4,13); 
				\filldraw[draw=black,fill=lightgray] (5,13)--(6,13)--(6,14)--(5,14)--(5,13); 
				\filldraw[draw=black,fill=lightgray] (6,13)--(7,13)--(7,14)--(6,14)--(6,13); 
				\filldraw[draw=black,fill=lightgray] (7,13)--(8,13)--(8,14)--(7,14)--(7,13); 
				\filldraw[draw=black,fill=lightgray] (8,13)--(9,13)--(9,14)--(8,14)--(8,13); 
				\filldraw[draw=black,fill=lightgray] (9,13)--(10,13)--(10,14)--(9,14)--(9,13); 
				\filldraw[draw=black,fill=lightgray] (10,13)--(11,13)--(11,14)--(10,14)--(10,13); 
				\filldraw[draw=black,fill=lightgray] (11,13)--(12,13)--(12,14)--(11,14)--(11,13); 
				\filldraw[draw=black,fill=lightgray] (12,13)--(13,13)--(13,14)--(12,14)--(12,13); 
				\filldraw[draw=black,fill=lightgray] (13,13)--(14,13)--(14,14)--(13,14)--(13,13); 
				\filldraw[draw=black,fill=lightgray] (0,12)--(1,12)--(1,13)--(0,13)--(0,12); 
				\filldraw[draw=black,fill=lightgray] (2,12)--(3,12)--(3,13)--(2,13)--(2,12); 
				\filldraw[draw=black,fill=lightgray] (3,12)--(4,12)--(4,13)--(3,13)--(3,12); 
				\filldraw[draw=black,fill=lightgray] (4,12)--(5,12)--(5,13)--(4,13)--(4,12); 
				\filldraw[draw=black,fill=lightgray] (5,12)--(6,12)--(6,13)--(5,13)--(5,12); 
				\filldraw[draw=black,fill=lightgray] (6,12)--(7,12)--(7,13)--(6,13)--(6,12); 
				\filldraw[draw=black,fill=lightgray] (7,12)--(8,12)--(8,13)--(7,13)--(7,12); 
				\filldraw[draw=black,fill=lightgray] (8,12)--(9,12)--(9,13)--(8,13)--(8,12); 
				\filldraw[draw=black,fill=lightgray] (9,12)--(10,12)--(10,13)--(9,13)--(9,12); 
				\filldraw[draw=black,fill=lightgray] (0,11)--(1,11)--(1,12)--(0,12)--(0,11); 
				\filldraw[draw=black,fill=lightgray] (2,11)--(3,11)--(3,12)--(2,12)--(2,11); 
				\filldraw[draw=black,fill=lightgray] (3,11)--(4,11)--(4,12)--(3,12)--(3,11); 
				\filldraw[draw=black,fill=lightgray] (4,11)--(5,11)--(5,12)--(4,12)--(4,11); 
				\filldraw[draw=black,fill=lightgray] (5,11)--(6,11)--(6,12)--(5,12)--(5,11); 
				\filldraw[draw=black,fill=lightgray] (6,11)--(7,11)--(7,12)--(6,12)--(6,11); 
				\filldraw[draw=black,fill=lightgray] (7,11)--(8,11)--(8,12)--(7,12)--(7,11); 
				\filldraw[draw=black,fill=lightgray] (8,11)--(9,11)--(9,12)--(8,12)--(8,11); 
				\filldraw[draw=black,fill=lightgray] (0,10)--(1,10)--(1,11)--(0,11)--(0,10); 
				\filldraw[draw=black,fill=lightgray] (2,10)--(3,10)--(3,11)--(2,11)--(2,10); 
				\filldraw[draw=black,fill=lightgray] (3,10)--(4,10)--(4,11)--(3,11)--(3,10); 
				\filldraw[draw=black,fill=lightgray] (4,10)--(5,10)--(5,11)--(4,11)--(4,10); 
				\filldraw[draw=black,fill=lightgray] (5,10)--(6,10)--(6,11)--(5,11)--(5,10); 
				\filldraw[draw=black,fill=lightgray] (6,10)--(7,10)--(7,11)--(6,11)--(6,10); 
				\filldraw[draw=black,fill=lightgray] (7,10)--(8,10)--(8,11)--(7,11)--(7,10); 
				\filldraw[draw=black,fill=lightgray] (8,10)--(9,10)--(9,11)--(8,11)--(8,10); 
				\filldraw[draw=black,fill=lightgray] (11,10)--(12,10)--(12,11)--(11,11)--(11,10); 
				\filldraw[draw=black,fill=lightgray] (12,10)--(13,10)--(13,11)--(12,11)--(12,10); 
				\filldraw[draw=black,fill=lightgray] (0,9)--(1,9)--(1,10)--(0,10)--(0,9); 
				\filldraw[draw=black,fill=lightgray] (2,9)--(3,9)--(3,10)--(2,10)--(2,9); 
				\filldraw[draw=black,fill=lightgray] (3,9)--(4,9)--(4,10)--(3,10)--(3,9); 
				\filldraw[draw=black,fill=lightgray] (4,9)--(5,9)--(5,10)--(4,10)--(4,9); 
				\filldraw[draw=black,fill=lightgray] (5,9)--(6,9)--(6,10)--(5,10)--(5,9); 
				\filldraw[draw=black,fill=lightgray] (6,9)--(7,9)--(7,10)--(6,10)--(6,9); 
				\filldraw[draw=black,fill=lightgray] (7,9)--(8,9)--(8,10)--(7,10)--(7,9); 
				\filldraw[draw=black,fill=lightgray] (0,8)--(1,8)--(1,9)--(0,9)--(0,8); 
				\filldraw[draw=black,fill=lightgray] (2,8)--(3,8)--(3,9)--(2,9)--(2,8); 
				\filldraw[draw=black,fill=lightgray] (3,8)--(4,8)--(4,9)--(3,9)--(3,8); 
				\filldraw[draw=black,fill=lightgray] (4,8)--(5,8)--(5,9)--(4,9)--(4,8); 
				\filldraw[draw=black,fill=lightgray] (0,7)--(1,7)--(1,8)--(0,8)--(0,7); 
				\filldraw[draw=black,fill=lightgray] (2,7)--(3,7)--(3,8)--(2,8)--(2,7); 
				\filldraw[draw=black,fill=lightgray] (3,7)--(4,7)--(4,8)--(3,8)--(3,7); 
				\filldraw[draw=black,fill=lightgray] (2,5)--(3,5)--(3,6)--(2,6)--(2,5); 
				\filldraw[draw=black,fill=lightgray] (3,5)--(4,5)--(4,6)--(3,6)--(3,5); 
				\filldraw[draw=black,fill=lightgray] (6,5)--(7,5)--(7,6)--(6,6)--(6,5); 
				\filldraw[draw=black,fill=lightgray] (7,5)--(8,5)--(8,6)--(7,6)--(7,5); 
				\filldraw[draw=black,fill=lightgray] (11,5)--(12,5)--(12,6)--(11,6)--(11,5); 
				\filldraw[draw=black,fill=lightgray] (2,4)--(3,4)--(3,5)--(2,5)--(2,4); 
				\filldraw[draw=black,fill=lightgray] (3,4)--(4,4)--(4,5)--(3,5)--(3,4); 
				\filldraw[draw=black,fill=lightgray] (6,4)--(7,4)--(7,5)--(6,5)--(6,4); 
				\filldraw[draw=black,fill=lightgray] (3,2)--(4,2)--(4,3)--(3,3)--(3,2); 
				
				\draw (0,1)--(15,1);
				\draw (0,2)--(15,2);
				\draw (0,3)--(15,3);
				\draw (0,4)--(15,4);
				\draw (0,5)--(15,5);
				\draw (0,6)--(15,6);
				\draw (0,7)--(15,7);
				\draw (0,8)--(15,8);
				\draw (0,9)--(15,9);
				\draw (0,10)--(15,10);
				\draw (0,11)--(15,11);
				\draw (0,12)--(15,12);
				\draw (0,13)--(15,13);
				\draw (0,14)--(15,14);
				
				\draw (1,0)--(1,15);
				\draw (2,0)--(2,15);
				\draw (3,0)--(3,15);
				\draw (4,0)--(4,15);
				\draw (5,0)--(5,15);
				\draw (6,0)--(6,15);
				\draw (7,0)--(7,15);
				\draw (8,0)--(8,15);
				\draw (9,0)--(9,15);
				\draw (10,0)--(10,15);
				\draw (11,0)--(11,15);
				\draw (12,0)--(12,15);
				\draw (13,0)--(13,15);
				\draw (14,0)--(14,15);
				
				\node[left] at (1.58, 14.5) {3};
				\node[left] at (2.58, 3.5) {2};
				\node[left] at (4.58, 7.5) {2};
				\node[left] at (5.58, 8.5) {2};
				\node[left] at (8.58, 9.5) {1};
				\node[left] at (9.58, 11.5) {1};
				\node[left] at (10.58, 12.5) {1};
				\node[left] at (14.58, 13.5) {0};
				
			\end{tikzpicture}
			\caption{A visual aid for the proof of Theorem \ref{thm:staircaseheavy}.}
			\label{fig:staircase}
		\end{center}
	\end{figure}
	
	\begin{lemma}
		\label{lem:heavyminfirstmarks}
		Let $C$ be a heavy climbing chain of $w\in S_n$. Then
		\[\dwt(C,M(C))_{i_1}\geq \dwt(\CS(w),M(\CS(w)))_{i_1}. \]
	\end{lemma}
	\begin{proof}
		By Lemma \ref{lem:hwminimal}, it is enough to show that 
		\[\dwt(\CS(w),\CS(w))_{i_1} = h(w). \]
		The chain $\CS(w)$ is constructed iteratively by adding chunks of the form $\stair(v)$ for $v\in S_n$. Within a chunk $\stair(v)$, the first components are all the same and the second components are decreasing. The last element $(a,b)$ in $\stair(v)$ will exactly be the highest hook nested under the hook in row $a$ of $D(v)$. The first link of the next chunk will either have first component $a'>a$, or first component $a$ and second component $b'>b$. Thus, the links in $M(\CS(w))$ are exactly the first elements in each chunk $\stair(v)$ used in constructing $\CS(w)$. The last links of the chunks are exactly the row indices of the hooks in $D(w)$ counted by $h(w)$.
	\end{proof}

	\begin{conjecture}
		\label{conj:staircaseleading}
		Fix $w\in S_n$ and pick any term order satisfying $x_1>x_2>\cdots>x_n$. Let $\xi=(\CS(w),M(\CS(w)))$. Then
		\[\bm{x}^{\dwt(\xi)} = \min \left\{\bm{x}^{\dwt(C,M(C))} \mid C \mbox{ is a heavy climbing chain of } w \right\}. \] 
		Equivalently,
		\[\bm{x}^{\wt(\xi)} = \max \left\{\bm{x}^{\wt(C,M(C))} \mid C \mbox{ is a heavy climbing chain of } w \right\}. \] 
	\end{conjecture}

	 By Theorem \ref{thm:leapingleading}, the leading monomial of $\mathfrak{S}_w$ in any term order with $x_1>x_2>\cdots>x_n$ is witnessed by $\CL(w)$. The following conjecture is an analogue of Conjecture \ref{conj:elena}/Theorem \ref{thm:leadingchainweight}. We have tested it for $n\leq 8$.
	\begin{conjecture}
		\label{conj:reverseinterpolating}
		Fix $w\in S_n$ and any term order with $x_1>x_2>\cdots>x_n$. For $\ell(w)< k\leq \absrajcode{w}$, let $m_k(\bm{x})$ be the leading monomial of the degree $k$ homogeneous component of $\mathfrak{G}_w$. Then
		\[m_k(\bm{x}) = x_pm_{k-1}(\bm{x}) \]
		where $p$ is the smallest index such that $x_p m_{k-1}(\bm{x})$ divides $\bm{x}^{\wt(\CS(w),M(\CS(w)))}$.
	\end{conjecture}

	\begin{example}
		Mirroring Example \ref{exp:5721463leads}, let $w=5721463$. Then
		\begin{align*}
			\wt(\CL(w),\CL(w)) &= (5, 4, 2, 1, 0, 0, 0),\\
			\wt(\CS(w),M(\CS(w))) &= (5, 5, 2, 1, 1, 1, 0).
		\end{align*}
		Then the conjectured leading monomials of the homogeneous components of $\mathfrak{G}_w$ have exponents
		\begin{align*}
			&(5, 4, 2, 1, 0, 0, 0),\\
			&(5, \textbf{5}, 2, 1, 0, 0, 0),\\
			&(5, 5, 2, 1, \textbf{1}, 0, 0),\\
			&(5, 5, 2, 1, 1, \textbf{1}, 0).\qedhere
		\end{align*}
	\end{example}

	\bibliographystyle{plain}
	\bibliography{grothendieck-bibliography}

\begin{thebibliography}{10}

\bibitem{laddermoves}
N.~Bergeron and S.~Billey.
\newblock R{C}-graphs and {S}chubert polynomials.
\newblock {\em Experiment. Math.}, 2(4):257--269, 1993.

\bibitem{BJS}
S.~Billey, W.~Jockusch, and R.~P. Stanley.
\newblock Some combinatorial properties of {S}chubert polynomials.
\newblock {\em J. Algebraic Combin.}, 2(4):345--374, 1993.

\bibitem{combcoxgroups}
A.~Bj\"orner and F.~Brenti.
\newblock {\em Combinatorics of {C}oxeter Groups}.
\newblock Graduate Texts in Mathematics. Springer-Verlag, Heidelberg, 1
  edition, 2005.

\bibitem{balancedtableaux}
S.~Fomin, C.~Greene, V.~Reiner, and M.~Shimozono.
\newblock Balanced labellings and {S}chubert polynomials.
\newblock {\em European J. Combin}, 18:373--389, 1997.

\bibitem{FKschub}
S.~Fomin and A.~N. Kirillov.
\newblock The {Y}ang-{B}axter equation, symmetric functions, and {S}chubert
  polynomials.
\newblock {\em Discrete Math.}, 153(1):123--143, 1996.
\newblock Proceedings of the 5th Conference on Formal Power Series and
  Algebraic Combinatorics.

\bibitem{nilcoxeter}
S.~Fomin and R.~P. Stanley.
\newblock Schubert polynomials and the nil{C}oxeter algebra.
\newblock {\em Adv. in Math.}, 103(2):196 -- 207, 1994.

\bibitem{elena1}
E.~S. Hafner.
\newblock Vexillary {G}rothendieck polynomials via bumpless pipe dreams, 2022.
\newblock {\arxiv{2201.12432}}.

\bibitem{thomas}
T.~Lam, S.~Lee, and M.~Shimozono.
\newblock Back stable {S}chubert calculus, Jun 2018.
\newblock {\arxiv{1806.11233}}.

\bibitem{LS1}
A.~Lascoux and M.-P. Sch\"utzenberger.
\newblock Polyn\^omes de {S}chubert.
\newblock {\em C. R. Acad. Sci. Paris S\'er. I Math.}, 294(13):447--450, 1982.

\bibitem{LS2}
A.~Lascoux and M.-P Sch\"{u}tzenberger.
\newblock Structure de {H}opf de l'anneau de cohomologie et de l'anneau de
  {G}rothendieck d'une vari\'{e}t\'{e} de drapeaux.
\newblock {\em C. R. Acad. Sci. Paris S\'{e}r. I Math.}, 295(11):629--633,
  1982.

\bibitem{lenart}
C.~Lenart.
\newblock A unified approach to combinatorial formulas for {S}chubert
  polynomials.
\newblock {\em J. Algebraic Combin.}, 20(3):263--299, 2004.

\bibitem{LRS}
C.~Lenart, S.~Robinson, and F.~Sottile.
\newblock Grothendieck polynomials via permutation patterns and chains in the
  {B}ruhat order.
\newblock {\em Amer. J. Math.}, 128(4):805--848, 2006.

\bibitem{manivel}
L.~Manivel.
\newblock {\em Symmetric functions, {S}chubert polynomials and degeneracy
  loci}, volume~6 of {\em SMF/AMS Texts and Monographs}.
\newblock American Mathematical Society, Providence, RI; Soci\'et\'e
  Math\'ematique de France, Paris, 2001.
\newblock Translated from the 1998 French original by John R. Swallow, Cours
  Sp\'ecialis\'es [Specialized Courses], 3.

\bibitem{grobner}
E.~Miller and A.~Knutson.
\newblock Gr\"obner geometry of {S}chubert polynomials.
\newblock {\em Ann. of Math.}, 161(3):1245--1318, 2005.

\bibitem{CCA}
E.~Miller and B.~Sturmfels.
\newblock {\em Combinatorial Commutative Algebra}.
\newblock Graduate Texts in Mathematics. Springer New York, 2004.

\bibitem{PSW}
O.~Pechenik, D.~Speyer, and A.~Weigandt.
\newblock Castelnuovo--{M}umford regularity of matrix {S}chubert varieties,
  2021.
\newblock {\arxiv{2111.10681}}.

\bibitem{symmgrothdeg}
J.~Rajchgot, Y.~Ren, C.~Robichaux, A.~St. Dizier, and A.~Weigandt.
\newblock Degrees of symmetric {G}rothendieck polynomials and
  {C}astelnuovo--{M}umford regularity.
\newblock {\em Proc. Amer. Math. Soc.}, 2020.
\newblock to appear.

\bibitem{vexgrothdeg}
J.~Rajchgot, C.~Robichaux, and A.~Weigandt.
\newblock Castelnuovo--{M}umford regularity of ladder determinantal varieties
  and patches of {G}rassmannian {S}chubert varieties, 2022.
\newblock {\arxiv{2202.03995}}.

\bibitem{annacascade}
A.~Weigandt.
\newblock The {C}astelnuovo-{M}umford regularity of matrix {S}chubert
  varieties.
\newblock Presented at Cascade Lectures in Combinatorics, 2021.

\bibitem{prismtableaux}
A.~Weigandt and A.~Yong.
\newblock The prism tableau model for {S}chubert polynomials.
\newblock {\em J. Comb. Theory, Ser. A}, 154:551--582, 2018.

\end{thebibliography}
\end{document}